\documentclass[11pt,reqno,tbtags]{amsart}

\usepackage{amsthm}
\usepackage{amssymb}
\usepackage{amsfonts, mathrsfs}
\usepackage{amsmath}
\usepackage[numbers, sort&compress]{natbib}
\usepackage{graphicx}
\usepackage{vmargin, enumerate}
\usepackage{setspace}
\usepackage{paralist}
\usepackage[pdftex,colorlinks=true,citecolor=linkgreen,linkcolor=linkgreen]{hyperref}
\usepackage{color}
\usepackage{wrapfig}
\usepackage{setspace}
\usepackage{bbm}

\numberwithin{equation}{section}

\newcommand{\e}[1]{{\mathbf E}\left[#1\right]}
\newcommand{\p}[1]{{\mathbf P}\left(#1\right)}
\newtheorem{thm}{Theorem}[section]
\newtheorem{lem}[thm]{Lemma}
\newtheorem{cor}[thm]{Corollary}
\newtheorem{prop}[thm]{Proposition}
\newtheorem{dfn}[thm]{Definition}

\newtheorem{rmk}{Remark}[section]

\newcommand{\nk}{\mathbf{n}_{\kappa}}

\newcommand{\sk}{\mathbf{s}_{\kappa}}
\newcommand{\ck}{c_{\kappa}}

\definecolor{darkgreen}{rgb}{0.05,0.45,0.1}
\definecolor{linkgreen}{rgb}{0,0.5,0}

\begin{document}

\title{Scaling limit of random forests with prescribed degree sequences}
\author{Tao Lei}
\address{Department of Mathematics and Statistics, McGill University, 805 Sherbrooke Street West,
		Montr\'eal, Qu\'ebec, H3A 0B9, Canada}
\email{tao.lei@mail.mcgill.ca}
\date{March 31, 2017} 



\begin{abstract}
In this paper, we consider the random plane forest uniformly drawn from all possible plane forests with a given degree sequence. Under suitable conditions on the degree sequences, we consider the limit of a sequence of such forests with the number of vertices tends to infinity in terms of Gromov-Hausdorff-Prokhorov topology. This work falls into the general framework of showing convergence of random combinatorial structures to certain Gromov-Hausdorff scaling limits, described in terms of the {\em Brownian Continuum Random Tree} (BCRT), pioneered by the work of Aldous \cite{Aldous19911, Aldous19912, Aldous1993}. In fact we identify the limiting random object as a sequence of random real trees encoded by excursions of some first passage bridges reflected at minimum. We establish such convergence by studying the associated Lukasiewicz walk of the degree sequences. In particular, our work is closely related to and uses the results from the recent work of Broutin and Marckert \cite{BroutinMarckert2012} on scaling limit of random trees with prescribed degree sequences, and the work of Addario-Berry \cite{Addario2012} on tail bounds of the height of a random tree with prescribed degree sequence.
 \end{abstract}

\maketitle


\section{Introduction}\label{sec:intro} 

Scaling limits for finite graphs is a topic at the intersection of combinatorics and probability. In this paper, we investigate the Gromov-Hausdorff-Prokhorov convergence of random forests with prescribed degree sequence. Our work is a natural continuation of \cite{BroutinMarckert2012} where it is shown that under natural hypotheses on the degree sequences, after suitable normalization, uniformly random trees with given degree sequence converge to Brownian continuum random tree, with the size of trees going to infinity.

In a series of papers \cite{Aldous19911, Aldous19912, Aldous1993}, Aldous introduced the concept of Brownian continuum random tree (BCRT) and showed that {\em critical} Galton-Watson tree conditioned on its size has BCRT as limiting objects. Since then, many families of graphs have been shown to have BCRT or random processes derived from BCRT as their limiting objects. For example, multi-type Galton-Watson trees \cite{Miermont 2008}, unordered binary trees \cite{MarckertMiermont2011}, critical Erd\"{o}s-R\'{e}nyi  random graph \cite{AddarioberryBroutinGoldschmidt10}, random planar maps with a unique large face \cite{JansonStefansson2015}, random planar quadrangulations with a boundary \cite{Bettinelli2015}.

As in \cite{BroutinMarckert2012}, our combinatorial model is motivated by the metric structure of graphs with a prescribed degree sequence. This model was first introduced by Bender and Canfield \cite{BenderCanfield} and by Bollob\'{a}s \cite{Bollobas} in the form of the configuration model. This model can give rise to graphs with any particular (legitimate) prescribed degree sequence (including, e.g., heavy tailed degree distributions, a feature which is observed in realistic networks but is not captured by the Erd\"{o}s-R\'{e}nyi random graph model). 

Our main results, which are stated formally in Section \ref{subsection:statement of main thm}, are that, under natural assumptions on degree sequences and after suitable normalization, large uniformly random forests with given degree sequence converge in distribution to the forests coded by Brownian first passage bridge, with respect to the Gromov-Hausdorff-Prokhorov topology. In order to present these results rigorously, we need the following subsection to introduce the necessary concepts and notations involved.

\subsection{Definitions and Notation}
\subsection*{Plane trees and forests}
We recall the following definition of plane trees (as in e.g. \cite{DuquesneLegall2002}). Let $$\mathcal{U}=\bigcup\limits_{n=0}^\infty \mathbb{N}^n,$$ where $\mathbb{N}=\{1, 2, \cdots \}$ and $\mathbb{N}^0=\{\emptyset\}$. If $u=(u_1, u_2, \cdots, u_n)\in\mathcal{U}$ we write $u=u_1u_2\cdots u_n$ for short and let $|u|=n$ be the {\em generation} of $u$. If $u=u_1\cdots u_m, v=v_1\cdots v_n$, we write $uv=u_1\cdots u_mv_1\cdots v_n$ for the {\em concatenation} of $u$ and $v$.

\begin{dfn}\label{def: plane tree}
A {\em rooted plane tree} $\mathrm{T}$ is a subset of $\mathcal{U}$ satisfying the following conditions:

(i)~$\emptyset\in \mathrm{T}$;

(ii)~If $v\in \mathrm{T}$ and $v=uj$ for some $u\in\mathcal{U}$ and $j\in \mathbb{N}$, then $u\in \mathrm{T}$;

(iii)~For every $u\in \mathrm{T}$, there exists a number $k_{\mathrm{T}}(u)\ge 0$ such that $uj\in \mathrm{T}$ if and only if $1\le j\le k_{\mathrm{T}}(u)$. We call $k_{\mathrm{T}}(u)$ the {\em degree} of $u$ in $\mathrm{T}$.
\end{dfn}

We denote the lexicographic order on $\mathcal{U}$ by $<$ (e.g. $\emptyset<11<21<22$). The lexicographic order on $\mathcal{U}$ induces a total order on the set of all rooted plane trees.

We call a finite sequence of finite rooted plane trees $\mathrm{F}=(\mathrm{T}_1, \mathrm{T}_2, \cdots, \mathrm{T}_m)$ a {\em rooted plane forest}. For forest $\mathrm{F}$, we let $\mathrm{F}^{\downarrow}$ be the sequence of tree components of $\mathrm{F}$ in decreasing order of size, breaking ties lexicographically (if again tied, then as the original order of appearance in $\mathrm{F}$).

\begin{dfn}
A {\em degree sequence} is a sequence $\mathbf{s}=(s^{(i)}, i\ge 0)$ of non-negative integers with $\sum\limits_{i\ge 0} s^{(i)}<\infty$ such that $c(\mathbf{s}):= \sum\limits_{i\ge 0}(1-i)s^{(i)}>0$.
For a plane tree $\mathrm{T}$, the degree sequence $\mathbf{s}(\mathrm{T})=(s^{(i)}(\mathrm{T}), i\ge 0)$ is given by $$s^{(i)}(\mathrm{T})=|\{u\in\mathrm{T}: k_{\mathrm{T}}(u)=i\}|.$$ 

For a plane forest $\mathrm{F}=\left(\mathrm{T}_1, \cdots, \mathrm{T}_m\right)$, the degree sequence $\mathbf{s}(\mathrm{F})=(s^{(i)}(\mathrm{F}), i\ge 0)$ is given by \[s^{(i)}(\mathrm{F})=\sum\limits_{j=1}^m s^{(i)}(\mathrm{T}_j).\] 
\end{dfn}

Note that $c(\mathbf{s}(\mathrm{T}))=1$ for any plane tree $\mathrm{T}$. In general since $$\sum\limits_{i\ge 0}i s^{(i)}(\mathrm{F})=\sum\limits_{j=1}^m\sum\limits_{u\in\mathrm{T}_j}k_{\mathrm{T}_j}(u)=\sum\limits_{j=1}^m(|\mathrm{T}_j|-1)$$ and $\sum\limits_{i\ge 0}s^{(i)}(\mathrm{F})=\sum\limits_{j=1}^m|\mathrm{T}_j|$, the number of tree components in $\mathrm{F}$ is always $c(\mathbf{s}(\mathrm{F}))$. For any degree sequence $\mathbf{s}$, we adopt the notations \[n(\mathbf{s}):=\sum\limits_{i\ge 0} s^{(i)},\ \ \  \Delta(\mathbf{s}):=\max\{i: s^{(i)}>0\}.\]

Figure \ref{fig:forest with degree sequence}, below, shows a plane forest with degree sequence $\mathbf{s}=(7,2,2,1,0,\cdots)$ with $s^{(i)}=0$ for $i\ge 4$.

\begin{figure}[h]\label{fig:forest with degree sequence}
\centering
\includegraphics[scale=0.9]{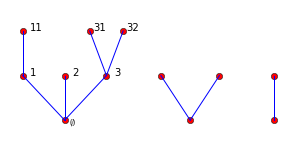}
\caption{A plane forest (with labels for the first tree) with degree sequence $\mathbf{s}=(7,2,2,1,0,\cdots)$}
\end{figure}

For any degree sequence $\mathbf{s}=(s^{(i)}, i\ge 0)$, we let $\mathrm{F}(\textbf{s})$ denote the set of all plane forests with degree sequence $\textbf{s}$. Let $\mathbb{P}_\mathbf{s}$ be the uniform measure on $\mathrm{F}(\textbf{s})$ and let $\mathbb{F}(\mathbf{s})$ be a random plane forest with law $\mathbb{P}_\mathbf{s}$.

%
%

\subsection*{First passage bridge}
We also need to recall the following definition of {\em first passage bridge} as in \cite{BCP2003}. Informally, for $\lambda>0$, the {\em first passage bridge} of unit length from 0 to $-\lambda$, denoted $F^{br}_\lambda$, is a $C[0, 1]-$valued random variable with law $$(F^{br}_{\lambda}(t), 0\le t\le 1)\overset{d}{=}(B(t),0\le t\le 1 ~|~ T_{\lambda}=1)$$ where $B$ is a standard Brownian motion and $T_{\lambda}:=\inf\{t:B(t)<-\lambda\}$ is the first passage time below level $-\lambda< 0$.

For $l\ge 0$, we write $B^{br}_l$ for the Brownian bridge of duration 1 from 0 to $-l$. As explained in Proposition 1 of \cite{FPY1992}, the law of the 
Brownian bridge $B^{br}_l$ is characterized by $B^{br}_l(1)=-l$ and the formula 
\begin{equation}\label{eqn:brownian bridge change of measure}
\e{f((B^{br}_l(t))_{0\le t\le m})}=\e{f((B(t))_{0\le t\le m})\frac{p_{1-m}(-l-B(m))}{p_1(-l)}}
\end{equation}
for all bounded measurable function $f$, and all $0\le m<1$, where $p_a$ is the Gaussian density with variance $a$ and mean 0, that is, $p_a(x)=\frac{1}{\sqrt{2\pi a}}e^{-\frac{x^2}{2a}}$. 
In a similar way the law of $F^{br}_{\lambda}$ can be defined as the law such that 
\begin{equation}\label{eqn:first passage bridge change of measure}
\e{f((F^{br}_{\lambda}(t))_{0\le t\le s})}=\e{(f(B(t))_{0\le t\le s})\frac{p'_{1-s}(-\lambda-B(s))}{p'_{1}(-\lambda)}{\mathbbm{1}}_{\{\inf\limits_{r\le s}B(r) >-\lambda\}}}
\end{equation}
for all bounded measurable functions $f$ and all $0\le s<1$ and $F^{br}_{\lambda}(1)=-\lambda$, where $p'_a$ is the derivative of $p_a$. These formulae set the finite-dimensional laws of the first passage bridge. In \cite{Bettinelli2010} (see Section 5.1 for details) it is shown that it admits a continuous version, and that $F^{br}_\lambda$ is the weak limit of $F^\epsilon_\lambda$ where $(F^\epsilon_\lambda(t), 0\le t\le 1)$ has the law of $B$ conditioned on the event $\{B(1)<-\lambda+\epsilon,~ \inf\limits_{s\le 1}B(s)>-\lambda-\epsilon\}$, hence justifying the informal conditioning definition.

\subsection*{Gromov-Hausdorff-Prokhorov distance}
We recall the definition of the Gromov-Hausdorff distance (see for example Definition 7.3.10 in \cite{Burago2001}). Let $(X, d)$ and $(X', d')$ be compact metric spaces. Then the {\em Gromov-Hausdorff distance} between $(X, d)$ and $(X', d')$ is given by $$d_{GH}((X,d),(X',d'))=\inf\limits_{\phi,\phi', Z}d_H^Z(\phi(X),\phi'(X')),$$ where the infimum is taken over all isometric embeddings $\phi:X\hookrightarrow Z$ and $\phi': X'\hookrightarrow Z$ into some common Polish metric space $(Z, d^Z)$ and $d_H^Z$ denotes the Hausdorff distance between compact subsets of $Z$, that is, $$d_H^Z(A,B)=\inf\{\epsilon>0: A\subset B^\epsilon, B\subset A^\epsilon\},$$ where $A^\epsilon$ is the {\em $\epsilon-$enlargement} of $A$: $$A^\epsilon=\{z\in Z: \inf\limits_{y\in A} d^Z(y, z)<\epsilon\}.$$

Note that strictly speaking $d_{GH}$ is not a distance since different compact metric spaces can have GH distance zero.

A {\em rooted measured metric space} $\mathcal{X}=(X, d, \emptyset, \mu)$ is a metric space $(X, d)$ with a distinguished element $\emptyset\in X$ and a finite Borel measure $\mu$. Note that the definitions in this subsection work in more general settings, e.g. $\mu$ could be a boundedly finite Borel measure (see \cite{AbrahamDH2013}), but for the purpose of this paper, finite measure $\mu$ is enough.

Let $\mathcal{X}=(X, d, \emptyset, \mu)$ and 
${\mathcal{X}}'=(X', d', \emptyset',\mu')$ be two compact rooted measured metric spaces, they are {\em GHP-isometric} if there exists an isometric one-to-one map $\Phi: X\rightarrow X'$ such that $\Phi(\emptyset)=\emptyset'$ and $\Phi_{\ast}\mu=\mu'$ where $\Phi_{\ast}\mu$ is the {\em push forward} of measure $\mu$ to $(X', d')$, that is, $\Phi_{\ast}\mu(A)=\mu(\Phi^{-1}(A))$ for $A\in\mathcal{B}(X')$. In this case, call $\Phi$ a {\em GHP-isometry}.

Suppose both $\mathcal{X}$ and $\mathcal{X'}$ are compact, then define the Gromov-Hausdorff-Prokhorov distance as:
$$d_{GHP}(\mathcal{X},\mathcal{X'})=\inf\limits_{\Phi, \Phi', Z}(d^Z(\Phi(\emptyset), \Phi'(\emptyset'))+d_H^Z(\Phi(X), \Phi'(X'))+d_P^Z(\Phi_\ast\mu, \Phi'_\ast\mu'))$$ where the infimum is taken over all isometric embeddings $\Phi: X\hookrightarrow Z$ and $\Phi': X'\hookrightarrow Z$ into some common Polish metric space $(Z, d^Z)$, and $d_P^Z$ denotes the Prokhorov distance between finite Borel measures on $Z$, that is, \[d_P^Z(\mu, \nu)=\inf\{\epsilon>0: \mu(A)\le\nu(A^\epsilon)+\epsilon, \nu(A)\le \mu(A^\epsilon)+\epsilon \mbox{ for any closed set }A\}.\]
Let $\mathbb{K}$ denote the set of GHP-isometry classes of compact rooted measured metric spaces and we identify $\mathcal{X}$ with its GHP-isometry class. We have the following results from \cite{AbrahamDH2013}:
\begin{thm}[Theorem 2.5 in \cite{AbrahamDH2013}]
The function $d_{GHP}$ defines a metric on $\mathbb{K}$ and the space $(\mathbb{K}, d_{GHP})$ is a Polish metric space.
\end{thm}
We next define a distance between sequences of rooted measured metric spaces. For $\mathbf{X}=(\mathcal{X}_j, j\ge 1), \mathbf{X}'=(\mathcal{X}'_j, j\ge 1)$ in $\mathbb{K}^{\mathbb{N}}$, we let
\[d_{GHP}^\infty(\mathbf{X}, \mathbf{X}')=\sup\limits_{j\ge 1} d_{GHP}(\mathcal{X}_j, \mathcal{X}'_j).\]
If $\mathbf{X}\in \mathbb{K}^n$ for some $n\in\mathbb{N}$, in order to view $\mathbf{X}$ as a member of $\mathbb{K}^\mathbb{N}$, we append to $\mathbf{X}$ an infinite sequence of zero metric spaces $\mathcal{Z}$. Here $\mathcal{Z}$ is the rooted measured metric space consisting of a single point with measure 0. Let $\mathbf{Z}=(\mathcal{Z}, \mathcal{Z},\cdots)$ and 
\[\mathbb{L}_\infty=\{\mathbf{X}\in \mathbb{K}^\mathbb{N}: \limsup\limits_{j\rightarrow\infty}d_{GHP}(\mathcal{X}_j,\mathcal{Z})=0\}.\]
By definition of GHP distance it is not hard to see that $d_{GHP}(\mathcal{X},\mathcal{Z})=\frac{\mathrm{diam}(X)}{2}+\mu(X)$, hence $\mathbf{X}\in\mathbb{L}_\infty$ if and only if $\limsup\limits_{j\rightarrow\infty}\left(\mathrm{diam}(X_j)+\mu_j(X_j)\right)=0$. It is likewise straightforward to show that $(\mathbb{L}_\infty, d_{GHP}^\infty)$ is a complete separable metric space.

%
%
%
%

\subsection*{Real trees}
Next we briefly recall the concepts of real trees and real trees coded by continuous functions. A more lengthy presentation about the probabilistic aspects of real trees can be found in \cite{Evans2008, Le Gall2005}.

\begin{dfn}
A compact metric space $(T, d)$ is a {\em real tree} if the following hold for every $a, b\in T$:

(i) There is a unique isometric map $f_{a,b}$ from $[0, d(a,b)]$ into $T$ such that $f_{a,b}(0)=a$ and $f_{a,b}(d(a,b))=b$.

(ii) If $q$ is a continuous injective map from $[0,1]$ into $T$, such that $q(0)=a$ and $q(1)=b$, we have $q([0,1])=f_{a,b}([0,d(a,b)])$.
\end{dfn}

A real tree $(T,d)$ is {\em rooted} if there is a distinguished vertex (the root) $\emptyset\in T$ and we denote a rooted real tree by $(T,d,\emptyset)$. If there is a finite Borel measure $\mu$ on $T$, then $(T, d, \emptyset, \mu)$ is a measured rooted real tree.

Next we show a way of constructing real trees from continuous functions. Let $g:[0,\infty)\rightarrow [0,\infty)$ be a continuous function with compact support and such that $g(0)=0$. For every $s, t\ge 0$, let $$d^\circ_g(s,t)=g(s)+g(t)-2m_g(s,t)$$ where $$m_g(s,t)=\min\limits_{s\wedge t\le r\le s\vee t} g(r).$$ The function $d^\circ_g$ is a pseudometric on $[0,\infty)$. Define an equivalence relation $\sim$ on $[0,\infty)$ by setting $s\sim t$ iff $d^\circ_g(s,t)=0$. Then let $T_g=[0,\infty)/\sim$ and let $d_g$ be the induced distance on $T_g$. Then $(T_g, d_g)$ is a real tree (see, e.g. Theorem 2.2 in \cite{Le Gall2005}).

To get an intuition of this construction, for a rooted plane tree $\mathrm{T}$ with graph distance $d_{gr}$, let $\hat{\mathrm{T}}$ be the metric space obtained from $\mathrm{T}$ by viewing each edge as an isometric copy of the unit interval $[0,1]$, and imagine a particle exploring the tree, starting from the root and moving at unit speed. Each time the particle leaves a vertex $u$, it moves to the lexicographically next unvisited child of $u$, if such a child exists; otherwise it moves to the parent of $u$. The exploration concludes the moment the particle has visited all vertices and returned to the root. Let $C:[0, 2(|\mathrm{T}|-1)]\rightarrow [0,\infty)$ be such that $C(t)$ equals to the graph distance between the particle and the root at time $t$. $C$ is called the {\em contour function} of $\mathrm{T}$. Then the metric space $\mathcal{T}_C$ constructed from $C$ is isometric to $\hat{\mathrm{T}}$.

Let $\emptyset_g$ denote the equivalence class of 0. Let $p_g$ be the canonical projection from $[0,\infty)$ to $T_g$ and $\sigma_g=sup\{t: g(t)>0\}$. Let $\textbf{m}_g$ be the push forward of the Lebesgue measure on $[0, \sigma_g]$ ($(\sigma_g,\infty)$ has measure 0) by $p_g$. Then $\mathcal{T}_g=(T_g, d_g, \emptyset_g, \textbf{m}_g)$ is a compact measured rooted real tree. In particular, $\mathcal{T}_g\in\mathbb{K}$. Let $\mathbf{e}$ denote the standard Brownian excursion, then $\mathcal{T}_{\mathbf{e}}$ is called the {\em Brownian continuum random tree} (BCRT for short).


\subsection{Statement of main theorems}\label{subsection:statement of main thm}
For $c>0$, let $c\mathbf{e}\in C[0,\infty)$ denote the Brownian excursion of length $c$, that is $(c\mathbf{e})(s):=\sqrt{c}\mathbf{e}(\frac{s}{c} \wedge 1)$ for $s\ge 0$. For any probability distribution $\textbf{p}=(p^{(i)}, i\ge 0)$ on $\mathbb{N}$, let $\mu(\textbf{p})=\sum\limits_{i\ge 0}i p^{(i)}$ and $\sigma^2(\textbf{p})=\sum\limits_{i\ge 0} i^2 p^{(i)}$.

In this paper we consider a sequence of degree sequences $(\mathbf{s}_\kappa, \kappa\in\mathbb{N})$, where $\mathbf{s}_\kappa=(s_\kappa^{(i)}, i\ge 0)$. We assume $\nk:=\sum\limits_{i\ge 0}s_\kappa^{(i)}\rightarrow\infty$ and let $\mathbb{F}_\kappa:=\mathbb{F}(\sk)$ and write $\mathbb{F}_\kappa^{\downarrow}=(\mathbb{T}_{\kappa, l},~ l\ge 1)$. We write $\textbf{p}_{\kappa}=(p_\kappa^{(i)}, i\ge 0):=(\frac{s^{(i)}_\kappa}{\nk}, i\ge 0)$. For $\mathbb{F}_\kappa^{\downarrow}=(\mathbb{T}_{\kappa, l},~ l\ge 1)$, let $\mathcal{T}_{\kappa,l}$ denote the measured rooted real tree $(\mathbb{T}_{\kappa, l}, \frac{\sigma_\kappa}{2\nk^{1/2}}d_{gr},\emptyset_{\kappa, l},\mu_{\kappa,l})$ where $\sigma_\kappa=\sigma(\textbf{p}_\kappa)$ and $\mu_{\kappa,l}$ denotes the uniform measure putting mass $\frac{1}{\nk}$ on each vertex of $\mathbb{T}_{\kappa, l}$. Let $\mathcal{F}_\kappa^\downarrow=(\mathcal{T}_{\kappa,l}, l\ge 1)$. Let $\Delta_{\kappa}:=\max\{i: s^{(i)}_{\kappa}>0\}$. We are now prepared to state our main theorems.

\begin{thm}\label{thm:forest of trees}
Suppose that there exists a distribution $\textbf{p}=(p^{(i)}, i\ge 0)$ on $\mathbb{N}$ with $p^{(1)}<1$ such that $\textbf{p}_\kappa$ converges to $\textbf{p}$ coordinatewise. Suppose also that $\sigma(\textbf{p}_\kappa)\rightarrow\sigma(\textbf{p})\in(0,\infty)$. If $\frac{c(\mathbf{s}_\kappa)}{\sigma(\mathbf{p}_\kappa)\nk^{1/2}}\rightarrow \lambda\in(0,\infty)$, then 
\begin{equation}\label{eqn:main_thm_statement}
\mathcal{F}_\kappa^\downarrow \xrightarrow{\mbox{d}}(\mathcal{T}_{\gamma_l}, l\ge 1) \mbox{ as } \kappa\rightarrow\infty,
\end{equation}
with respect to the product topology for $d_{GHP}$ where $(\gamma_l, l\ge 1)$ are the excursions of the process $(F_{\lambda}(s)-\inf\limits_{s'\in(0,s)}  F_{\lambda}(s'))_{0\le s\le 1}$, listed in decreasing order of length. 
\end{thm}

\begin{thm}\label{thm:forest of trees_2}
Under the conditions of Theorem \ref{thm:forest of trees}, suppose additionally that there exists $\epsilon>0$ such that $\Delta_\kappa=O(\nk^{\frac{1-\epsilon}{2}})$. Then the convergence (\ref{eqn:main_thm_statement}) holds in $(\mathbb{L}_\infty, d_{GHP}^\infty)$.
\end{thm}

\begin{rmk}\label{rmk:degree assumption}
The assumptions of Theorem \ref{thm:forest of trees} imply that $\mu(\textbf{p}_\kappa)\rightarrow \mu(\textbf{p})=1$ and that $\Delta_{\kappa}=o({\nk}^{1/2})$. We include the proof of these facts as Lemma \ref{lem:appendix degree assumption} in the Appendix.
\end{rmk}



\begin{rmk}
The pair $((\gamma_l, l\ge 1), (\mathcal{T}_{\gamma_l}, l\ge 1))$ has the same law as $((\gamma_l, l\ge 1),(\mathcal{T}_{|\gamma_l|\mathbf{e}_l}, l\ge 1))$ where $(\mathbf{e}_l,l\ge 1)$ are standard Brownian excursions, independent of each other and of $(\gamma_l, l\ge 1)$.
\end{rmk}

\subsection{Key ingredients of the paper}
Here we summarize the two key ingredients of this paper. The first element is the convergence of the large trees in (\ref{eqn:main_thm_statement}), which is essentially given by the following proposition. For all $l\ge 1$, let $X_{\kappa, l}=\frac{|\mathbb{T}_{\kappa,l}|}{\nk}$. 
\begin{prop}\label{prop: first j converge}
Under the conditions of Theorem \ref{thm:forest of trees}, for any fixed $j\ge 1$,
\begin{equation}\label{prop:large_tree_conv}((X_{\kappa,l})_{l\le j}, (\mathcal{T}_{\kappa,l})_{l\le j})\overset{d}{\rightarrow}((|\gamma_l|)_{l\le j}, (\mathcal{T}_{|\gamma_l|\mathbf{e}_l})_{l\le j})
\end{equation}
as $\kappa\rightarrow\infty$, where $(\mathbf{e}_l)_{l\le j}$ are independent copies of $\mathbf{e}$, and $(\gamma_l, l\ge 1)$ are the excursions of $(F_{\lambda}(s)-\inf\limits_{s'\in(0,s)}  F_{\lambda}(s'))_{0\le s\le 1}$ ranked in decreasing order of length.
\end{prop}

There are two parts of the convergence in (\ref{prop:large_tree_conv}). One is the convergence of the normalized sizes of large trees to lengths of excursions. This will be given by the following proposition. To state this result, we need to first introduce some notions. Let $\mathcal{C}_0(1)=\{x\in C([0,1], \mathbb{R}): x(0)=0\}$
For a non-negative function $g^+\in \mathcal{C}_0(1)$, an {\em excursion} $\gamma$ of $g^+$ is the restriction of $g^+$ to a time interval $[l(\gamma), r(\gamma)]$ such that $g^+(l(\gamma))=g^+(r(\gamma))=0$ and $g^+(s)>0$ for $s\in (l(\gamma), r(\gamma))$. In this case $[l(\gamma), r(\gamma)]$ is called an {\em excursion interval} of $g^+$. The length of the excursion is denoted as $|\gamma|=r(\gamma)-l(\gamma)$. For a function $g$ we write $g(s)-\min\limits_{0\le s' < s}g(s')$ to denote $(g(s)-\min\limits_{0\le s' < s}g(s'),~ 0\le s\le 1)$. For $g \in \mathcal{C}_0(1)$, sometimes we refer the excursions of $g(s)-\min\limits_{0\le s' < s}g(s')$ as excursions of $g$. Let $l^{\downarrow}_1=\{x=(x_1, x_2, \cdots): x_1\ge x_2\ge \cdots \ge 0, \sum\limits_i x_i\le 1\}$ and endow $l^\downarrow_1$ with the topology induced by the $l_1$ distance: $d(x,y)=\sum\limits_i |x_i-y_i|$. 

\begin{prop}\label{prop:length convergence}
Under the hypothesises of Theorem \ref{thm:forest of trees}, we have
\begin{equation}\label{eqn: tree sizes convergence}
(|\mathbb{T}_{\kappa, l}|/\nk)_{l\ge 1}\overset{d}{\rightarrow}(|\gamma_l|)_{l\ge 1}
\end{equation}
 in $l^{\downarrow}_1$, where $(\gamma_l, l\ge 1)$ are the excursions of $F_{\lambda}^{br}(s)-\min\limits_{0\le s'\le s} F_{\lambda}^{br}(s')$ ranked in decreasing order of length. 
\end{prop}

This proposition will be a corollary of the following theorem, which is the main result of Section \ref{sec:convergence of walk}. For a plane forest $\mathrm{F}$, let $u_1<u_2<\cdots<u_{|\mathrm{F}|}$ be the nodes of $\mathrm{F}$ listed according to their lexicographical order in $\mathcal{U}$ in each tree component, with nodes of first tree listed first, then the nodes of second tree and so on. The \emph{depth-first walk, or Lukasiewicz path} $S_\mathrm{F}$ is defined as follows. First set $S_\mathrm{F}(0)=0$ and then let $$S_\mathrm{F}(i)=\sum\limits_{j=1}^i(k_{\mathrm{F}}(u_j)-1) \mbox{ for } i=1, 2, \cdots, |\mathrm{F}|.$$ We extend the definition of $S_\mathrm{F}$ to the compact interval $[0, |\mathrm{F}|]$ by linear interpolation.

\begin{thm}\label{thm:walk convergence}
Under the conditions of Theorem \ref{thm:forest of trees}, we have
\begin{equation}\label{eqn: walk convergence}
\left(\frac{S_{\mathbb{F}_\kappa}(t\nk)}{\sigma(\mathbf{p}_\kappa){\nk}^{1/2}}\right)_{t\in[0,1]}\overset{d}{\rightarrow} F^{br}_{\lambda}
\end{equation}
in $\mathcal{C}_0(1)$ as $\kappa\rightarrow\infty$.
\end{thm}

The second part of the convergence of (\ref{prop:large_tree_conv}) is the convergence of the large trees, for which we will rely on the following result about random trees with given degree sequences from \cite{BroutinMarckert2012}.
\begin{thm}[Theorem 1 in \cite{BroutinMarckert2012}]\label{thm:Broutin-Marckert}
Let $\{\sk,\kappa\ge 1\}$ be a degree sequence such that $\nk:=n(\sk)\rightarrow\infty, \Delta_{\kappa}:=\Delta(\sk)=o(\nk^{1/2})$. Suppose that there exists a distribution $\textbf{p}$ on $\mathbb{N}$ with mean 1 such that $\textbf{p}_\kappa$ converges to $\textbf{p}$ coordinatewise and such that $\sigma(\textbf{p}_\kappa)\rightarrow\sigma(\textbf{p})\in(0,\infty)$. Let $\mathbb{T}_\kappa$ be the random plane tree under $\mathbb{P}_{\sk}$, the uniform measure on the set of plane trees with degree sequence $\sk$. Let $\mathcal{T}_\kappa$ denote the measured rooted metric space $(\mathbb{T}_\kappa, \frac{\sigma(\textbf{p}_\kappa)}{2\nk^{1/2}}d_{gr}, \emptyset_\kappa, \mu_\kappa)$ where $\mu_\kappa$ denotes the uniform measure putting mass $\frac{1}{\nk}$ on each vertex of $\mathbb{T}_\kappa$. Then when $\kappa\rightarrow\infty, \mathcal{T}_\kappa\overset{d}{\rightarrow}\mathcal{T}_{\mathbf{e}}$ in the Gromov-Hausdorff-Prokhorov sense. 
\end{thm}

\begin{rmk}\label{rmk:GHP justification}
In fact Theorem 1 in \cite{BroutinMarckert2012} is only stated in the Gromov-Hausdorff sense, that is, $(\mathbb{T}_\kappa, \frac{\sigma(\textbf{p}_\kappa)}{2\nk^{1/2}}d_{gr})\overset{d}{\rightarrow} (T_{\mathbf{e}}, d_{\mathbf{e}},\emptyset_{\mathbf{e}})$. But the conclusion can be strengthened to GHP convergence easily. For completeness, we include a proof of this fact in Appendix \ref{sec:GHP justification}.
\end{rmk}

The following proposition contains the additional ingredient required to prove Theorem \ref{thm:forest of trees_2}.

\begin{prop}\label{prop:diameter of small trees}
Under the conditions of Theorem \ref{thm:forest of trees_2}, for all $a>0$, we have
\[\lim\limits_{j\to\infty}\limsup\limits_{\kappa\to\infty} \p{\sup\limits_{l>j} \mathrm{diam}(\mathcal{T}_{\kappa,l})>a}=0.\]
\end{prop}

The key results leading to Proposition \ref{prop:diameter of small trees} include a height bound for random tree with prescribed degree sequence and a variance bound for uniformly permuted child sequences. The height bound of uniformly random tree with prescribed degree sequence is given in the following theorem.

\begin{thm}[Theorem 1 in \cite{Addario2012}]\label{thm:louigi_tree_height_tail_bound}
Fix a degree sequence $\mathbf{s}=(s^{(i)}, i\ge 0)$ such that $\sum\limits_{i\ge 0} is^{(i)}=|\mathbf{s}|-1$, and let $\mathbb{T}(\mathbf{s})$ be a uniformly random plane tree with degree sequence $\mathbf{s}$. Then for all $m\ge 1$ we have
\[\p{h(\mathbb{T}(\mathbf{s}))\ge m}\le 7\exp\left(-m^2/608\sigma^2(\mathbf{s})1_{\mathbf{s}}^2\right)\] where $1_{\mathbf{s}}=\frac{|\mathbf{s}|-2}{|\mathbf{s}|-1-s^{(1)}}$.
\end{thm}

The following probability bound on variances of uniformly permuted integer sequences allows us to control the variance of degrees of trees in random forests, and thereby  apply Theorem \ref{thm:louigi_tree_height_tail_bound} to prove Proposition \ref{prop:diameter of small trees}.


\begin{prop}\label{prop: concentration on long intervals}
Fix $c=(c_1, \cdots, c_n)\in \mathbb{N}^n$ and let $\pi$ be a uniformly random permutation of $\{1, \cdots, n\}$. Set $C_i=c_{\pi(i)}$ for $1\le i\le n$, and let $S_j=\sum\limits_{i\le j}C^2_i$ for $1\le i\le n$. Then for all $\lambda\ge 2$ and $1\le k\le n$, with $\Delta=\max\limits_{1\le i\le n}C_i=\max\limits_{1\le i\le n}c_i$, and $\sigma^2(c)=\sum\limits_{i\le n} c^2_i=S_n$, we have \[\p{S_k\ge\lambda\frac{k}{n}S_n}\le\exp\left(-\frac{3\sigma^2(c)}{16n}\cdot\frac{\lambda k}{\Delta^2}\right).\]

\end{prop}

Now let us prove our main theorems with these key results.
\begin{proof}[Proof of Theorem \ref{thm:forest of trees} and Theorem \ref{thm:forest of trees_2}]
By Skorokhod's representation theorem, we may work in a probability space in which the convergence in Proposition \ref{prop: first j converge} is almost sure. Hence Proposition \ref{prop: first j converge} yields that for any fixed $j, ~\sup\limits_{l\le j}d_{GHP}(\mathcal{T}_{\kappa,l}, \mathcal{T}_{|\gamma_l|\mathbf{e}_l})\overset{d}{\to} 0$. This establishes Theorem \ref{thm:forest of trees}.
Now to prove the convergence in $(\mathbb{L}_{\infty}, d_{GHP}^\infty)$, it suffices to prove 
that for any $a>0$, $$\lim\limits_{j\to\infty}\limsup\limits_{\kappa\to\infty} \p{\sup\limits_{l>j} \left(\mathrm{diam}(\mathcal{T}_{\kappa,l})+\mathrm{mass}(\mathcal{T}_{\kappa,l})+\mathrm{diam}(\mathcal{T}_{\gamma_l})+\mathrm{mass}(\mathcal{T}_{\gamma_l})\right)>a}=0.$$
It suffices to separately prove \[\lim\limits_{j\to\infty}\limsup\limits_{\kappa\to\infty} \p{\sup\limits_{l>j} \mathrm{diam}(\mathcal{T}_{\kappa,l})>a}=0, ~  \lim\limits_{j\to\infty}\limsup\limits_{\kappa\to\infty} \p{\sup\limits_{l>j} \mathrm{mass}(\mathcal{T}_{\kappa,l})>a}=0 \]
\[\lim\limits_{j\to\infty} \p{\sup\limits_{l>j} \mathrm{diam}(\mathcal{T}_{\gamma_l})>a}=0, ~ \lim\limits_{j\to\infty} \p{\sup\limits_{l>j} \mathrm{mass}(\mathcal{T}_{\gamma_l})>a}=0.\]

For this purpose, we need to control the probability that small trees having either large diameter or large mass. Note that for a tree its diameter is bounded by twice of its height. 

In fact the mass of tree is easy to control since for any $a>0$ and any $\kappa$,
\begin{eqnarray*}
\p{\sup_{l>j} \mathrm{mass}(\mathcal{T}_{\kappa,l})>a}&=& \p{\sup_{l>j} \frac{|\mathbb{T}_{\kappa,l}|}{\nk}>a}\\
&\le& \p{|\mathbb{T}_{\kappa,j}|>a\nk}=0 \mbox{ for } j>1/a
\end{eqnarray*}
For the diameter we resort to Proposition \ref{prop:diameter of small trees}.

We also need to bound $\mathrm{diam}(\mathcal{T}_{\gamma_l})$ and $\mathrm{mass}(\mathcal{T}_{\gamma_l})$ for $l$ large. Note that $\mathrm{mass}(\mathcal{T}_{\gamma_l})=|\gamma_l|$ and for any $a$, let $j> 1/a$, then $\p{\sup\limits_{l>j}|\gamma_l|>a}=0$.

For $\mathrm{diam}(\mathcal{T}_{\gamma_l}),~ \mathrm{diam}(\mathcal{T}_{\gamma_l})\le 2h(\mathcal{T}_{\gamma_l})=2\max(\gamma_l)$. For $0\le s\le 1$, let \[R(s)=F_{\lambda}(s)-\inf\limits_{s'\in(0,s)}  F_{\lambda}(s')\] and the excursion interval of $\gamma_l$ be $[g_l, d_l]$. Then 
\begin{eqnarray*}
\mathrm{diam}(\mathcal{T}_{\gamma_l}) &\le& 2\sup\limits_{t\in[g_l, d_l]}R(t)= 2 (\sup\limits_{t\in[g_l, d_l]} F_\lambda(t)-\inf\limits_{t\in[g_l, d_l]} F_\lambda(t))\\
&\le& 2\sup\left(|F_\lambda(t)-F_\lambda(s)|: |t-s|\le d_l-g_l\right)
\end{eqnarray*}
and $d_l-g_l=|\gamma_l|\le 1/l$.
So for any $j\ge 1/\epsilon$, \[\sup\limits_{l>j}\mathrm{diam}(\mathcal{T}_{\gamma_l})\le 2\sup\left(|F_\lambda(t)-F_\lambda(s)|: |t-s|\le \epsilon\right)\rightarrow 0 \mbox{ as } \epsilon\rightarrow 0\] since $F_\lambda$ is uniformly continuous. Hence we have the tail insignificance for diameter of $\mathcal{T}_{\gamma_l}$ and the claim is proved.
%
\end{proof}

To conclude this section, we sketch how our paper is organized. In Section \ref{sec:nto1} we investigate a special rotation mapping, which connects the collection of lattice bridges corresponding to certain degree sequence $\mathbf{s}$ and the set of first passage lattice bridges corresponding to $\mathbf{s}$. This will be the key starting point of our work using depth-first walk process to code the structure of random forests with given degree sequences. The combinatorial argument in this section will be also useful for our later work on transferring results such as Proposition \ref{prop: concentration on long intervals} to something similar which is applicable to random forests. This section will be purely combinatorial and only deal with fixed degree sequences. In Section \ref{sec:martingale stuff}, we collect some concentration results using martingale methods. These probability bounds will be useful for checking that the assumptions in Theorem \ref{thm:Broutin-Marckert} are satisfied for large trees of $\mathcal{F}^{\downarrow}_\kappa$. The second part of this section proves the variance bound in Proposition \ref{prop: concentration on long intervals}. Again all results in this section is non-asymptotic and hence are presented with regards to a fixed degree sequence. In Section \ref{sec:convergence of walk}, we prove Theorem \ref{thm:walk convergence}, the convergence of scaled exploration processes to some random process related to first passage bridge, using the rotation mapping in Section \ref{sec:nto1}. We will then get Proposition \ref{prop:length convergence} as a corollary from this weak convergence result. Finally, in Section \ref{sec:proof of prop6 and lem10} we finish the proof of Proposition \ref{prop: first j converge} and Proposition \ref{prop:diameter of small trees} using results from Section \ref{sec:martingale stuff} and Section \ref{sec:convergence of walk}.

\section{An $n-$to$-1$ map transforming lattice bridge to first passage lattice bridge}\label{sec:nto1}

Given a degree sequence $\textbf{s}=(s^{(i)}, i\ge 0)$, let $d(\textbf{s})\in\mathbb{Z}_{\ge 0}^{n(\textbf{s})}$ be the vector whose entries are weakly increasing and with $s^{(i)}$ entries equal to $i$, for each $i\ge 0$. For example, if $\textbf{s}=(3, 2, 0, 1, 0, \cdots)$ with $s^{(i)}=0$ for $i\ge 4$, then $d(\textbf{s})=(0, 0, 0, 1, 1, 3)$. 
Let $\mathrm{D}(\mathbf{s})$ be the collection of all possible child sequences corresponding to degree sequence $\textbf{s}$, i.e., all possible result as a permutation of $d(\textbf{s})$.

A {\em lattice bridge} is a function $b: [0, k]\rightarrow \mathbb{R}$ with $b(0)=0$ and $b(i)\in\mathbb{Z},\ \forall i\in [k]$, which is piecewise linear between integers. Here $k$ is an arbitrary positive integer. We let \[\Lambda(\textbf{s}) = \{b:[0, n(\textbf{s})]\rightarrow \mathbb{R}: b \mbox{ is a lattice bridge and } \forall i\ge 0, |\{j\in\mathbb{N}: b(j+1)-b(j)=i-1\}|= s^{(i)}\}\] and call $\Lambda(\textbf{s})$ the set of {\em lattice bridges corresponding to }$\textbf{s}$. Note that if $b\in\Lambda(\textbf{s})$, then $b(n(\textbf{s}))=-c(\textbf{s})$. Furthermore, we have $$|\Lambda(\textbf{s})|= {n \choose (s^{(i)}, i\ge 0)}=\frac{n!}{\prod\limits_{i\ge 0}s^{(i)}!}$$ since to determine $b\in\Lambda(\textbf{s})$, it suffices to choose the $s^{(0)}$ positions with step size $-1$, $s^{(1)}$ positions with step size 0, $s^{(2)}$ positions with step size 1, etc.

We then let $$F(\textbf{s})=\{b\in\Lambda(\textbf{s}): \inf\limits_{j\le n(\textbf{s})-1} b(j) >-c(\textbf{s})\}$$ and call $F(\textbf{s})$ the collection of {\em first passage lattice bridges corresponding to }$\textbf{s}$.

For $s>0$, let $\mathcal{C}_0(s)=\{x\in C([0,s], \mathbb{R}): x(0)=0\}$. For $u\in[0,s]$, let $\theta_{u,s}:\mathcal{C}_0(s)\rightarrow \mathcal{C}_0(s)$ denote the {\em cyclic shift at $u$}, that is,
$$(\theta_{u,s}(x))(t)=
\left\{
    \begin{array}{ll}
      x(t+u)-x(u), & \hbox{if $t+u\le s$;} \\
      x(t+u-s)+x(s)-x(u), & \hbox{if $t+u\ge s$.}
    \end{array}
  \right.
$$

For $x\in\mathcal{C}_0(s)$ and $y\in\mathbb{R}^-$, let $t(y,x):=\inf\{t\in[0,s]: x(t)\le y\}$ be the first time the graph of $x$ drops below $y$. Sometimes we drop the argument $x$ for convenience and simply write $t(y)$. If $y< \min\limits_{u\in[0,s]}x(u)$ we set $t(y,x)=0$ by convention, so $\theta_{t(y)}(x)=x$.

In what follows, for $k\in\mathbb{N}$ we write $[k]-1=\{0,1,\cdots, k-1\}$. And when the context is clear, we simply drop the subscript $s$ and write $\theta_u$ for $\theta_{u,s}$.

\begin{lem}\label{lem:f_g_j}
For $b\in \Lambda(\textbf{s})$, and for each $j\in [c(\textbf{s})]-1$, we have $\theta_{t(\min(b)+j)}(b)\in F(\textbf{s})$.
\end{lem}

\begin{proof}
Let $m\le 0$ be the minimum of $b$. Fix an integer $i$ such that $m\le i\le m+c(\textbf{s})-1$ and $u< n(\textbf{s})$. We shall prove that $\theta_{t(i)}(b)(u)>-c(\textbf{s})$, which proves the lemma.
If $0\le u\le n(\textbf{s})-t(i)$, then $\theta_{t(i)}(b)(u)=b(t(i)+u)-b(t(i))\ge m-i>-c(\textbf{s})$. If $n(\textbf{s})-t(i)\le u< n(\textbf{s})$, then $\theta_{t(i)}(b)(u)=b(t(i)+u-n(\textbf{s}))+b(n(\textbf{s}))-b(t(i))=b(t(i)+u-n(\textbf{s}))-c(\textbf{s})-i$. Since $u<n(\textbf{s})$, $t(i)+u-n(\textbf{s})<t(i)$ and we must have $b(t(i)+u-n(\textbf{s}))>i$ by our definition of $t$. Therefore in this case we also have $\theta_{t(i)}(b)(u)>-c(\textbf{s})$. 
\end{proof}

Next, define a function $f: \Lambda(\textbf{s})\times([c(\textbf{s})]-1)\rightarrow F(\textbf{s})$ by $f(b, j):=\theta_{t(\min(b)+j)}(b)$.


\begin{lem}\label{lem:nto1}
$f$ is an $n(\textbf{s})-$to$-1$ map from $\Lambda(\textbf{s})\times ([c(\textbf{s})]-1)$ to $F(\textbf{s})$.
\end{lem}

\begin{proof}
For $l\in F(\textbf{s})$, if size of preimage of $l$ under $f$ is strictly large than $n(\textbf{s})$, then we must have $b_1, b_2\in \Lambda(\textbf{s}), j_1, j_2\in [c(\textbf{s})]-1$ such that $f(b_1, j_1)=f(b_2, j_2)=l$ and $t(\min(b_1)+j_1)=t(\min(b_2)+j_2)$, since $t$ can only take values in $[n(\textbf{s})]$. By the definition of $f$ we must then have $b_1=b_2$ and hence $j_1=j_2$. Therefore each element in $F(\textbf{s})$ can have at most $n(\textbf{s})$ preimages in $\Lambda(\textbf{s})\times ([c(\textbf{s})]-1)$. On the other hand, we have (see, e.g., \cite{Pitman2006}, page 128) 
\begin{equation}\label{eqn:number of plane forests}
|F(\textbf{s})|=\frac{c(\textbf{s})}{n(\textbf{s})}{n(\textbf{s}) \choose (s^{(i)}, i\ge 0)}=\frac{c(\textbf{s})}{n(\textbf{s})}\frac{n(\textbf{s})!}{\prod\limits_{i\ge 0}s^{(i)}!}.
\end{equation}
Hence $n(\textbf{s})\times |F(\textbf{s})|=c(\textbf{s}) \times|\Lambda(\textbf{s})|=|\Lambda(\textbf{s})\times([c(\textbf{s})]-1)|$, so it must in fact hold that each $l\in F(\textbf{s})$ has exactly $n(\textbf{s})$ preimages.
\end{proof}

Recall the concept of depth-first walk $S_{\mathrm{F}}$ of a plane forest $\mathrm{F}$. For a sequence $\mathbf{c}=(c_1, \cdots, c_n)\in\mathbb{R}^n$, we write $W_{\mathbf{c}}(j)=\sum\limits_{i=1}^j(c_i-1)$ for $j\in [n]$. We let $W_{\mathbf{c}}(0)=0$ and make $W_{\mathbf{c}}$ a continuous function on $[0, n]$ by linear interpolation. Note that $S_{\mathrm{F}}$ is precisely $W_{\mathbf{c}}$ where $\mathbf{c}=(k_{\mathrm{F}}(u_1), \cdots, k_{\mathrm{F}}(u_{|\mathrm{F}|}))$.

For $\mathbf{c}=(c_1, \cdots, c_n)\in\mathbb{R}^n$ and a permutation $\pi$ of $[n]$, write $\pi(\mathbf{c})=(c_{\pi(1)}, \cdots, c_{\pi(n)})$. Also, recall from the beginning of this section that for a degree sequence $\textbf{s}$, $d(\textbf{s})$ is a vector with $s^{(i)}$ entries equal to $i$ for each $i\ge 0$.

\begin{cor}\label{cor:dfswalkofforest}
Let $\textbf{s}$ be a degree sequence. Let $\pi$ be a uniformly random permutation of $[n(\textbf{s})]$ and let $\nu$ be independent of $\pi$ and drawn uniformly at random from $[c(\textbf{s})]-1$.  Then $$f(W_{\pi(d(\textbf{s}))}, \nu)\overset{d}{=}S_{\mathbb{F}(\mathbf{s})},$$ and both are uniformly random elements of $F(\textbf{s})$.
\end{cor}

\begin{proof}
By definition, $(W_{\pi(d(\textbf{s}))}, \nu)$ is uniformly at random in $\Lambda(\textbf{s})\times([c(\textbf{s})]-1)$. By Lemma \ref{lem:nto1}, it follows that $f(W_{\pi(d(\textbf{s}))}, \nu)$ is uniformly random in $F(\textbf{s})$. On the other hand, the map sending plane forest $\mathrm{F}$ to its Lukasiewicz path $S_\mathrm{F}$ restricts to an invertible map from $\mathrm{F}(\textbf{s})$ to $F(\textbf{s})$. Thus, $S_{\mathbb{F}(\mathbf{s})}$ is also uniformly distributed in $F(\textbf{s})$.
\end{proof}

First-passage bridges are naturally connected to plane forests. In a similar way, general lattice bridges are naturally connected to {\em marked} plane forests. This interpretation will be more convenient for some later proofs (Propositions \ref{prop: bad_characterization}, \ref{prop:tree second momonet concentration alpha level} and \ref{prop:key_prop}).

A {\em marked forest} is a pair $(F, v)$ where $F$ is a plane forest and $v\in v(F)$. Sometimes we refer $v$ as the {\em mark} of $(F,v)$.
Recall that $\mathrm{F}(\textbf{s})$ denotes the collection of all plane forests with degree sequence $\textbf{s}$. Let $\mathrm{MF}(\textbf{s})$ be the collection of all marked forests with degree sequence $\textbf{s}$ and for $1\le i\le c(\textbf{s})$, let $\mathrm{MF}^i(\textbf{s})$ be the collection of marked forests $(F,v)\in \mathrm{MF}(\mathbf{s})$ such that the mark $v$ lies within the $i-$th tree of $F$. We define a map $g: \mathrm{MF}(\textbf{s})\rightarrow \mathrm{D}(\textbf{s})$ which lists the degrees of vertices of a marked forest starting from the mark in DFS order. Formally, for $(F, v)\in \mathrm{MF}(\textbf{s})$, if the DFS ordering of $v(F)$ is $v_1, \cdots, v_{n(\textbf{s})}$ and $v=v_i$, then $g((F,v))=(k_{F}(v_i), \cdots, k_{F}(v_{n(\textbf{s})}), k_{F}(v_1), \cdots, k_{F}(v_{i-1}))$. Next define a map $h: \mathrm{MF}(\textbf{s})\rightarrow\mathrm{F}(\textbf{s})$ by $h((F,v))=F$. Then we have the following easy fact.

\begin{lem}\label{lem c_to_1 and n_to_1}
$g$ is a $c(\textbf{s})-$to$-1$ surjective map and for each $1\le i\le c(\mathbf{s}),~g^i:=g\vert_{\mathrm{MF}^i(\mathbf{s})}$ is a bijection between $\mathrm{MF}^i(\mathbf{s})$ and $\mathrm{D}(\mathbf{s})$. Also, $h$ is a $n(\textbf{s})-$to$-1$ surjective map. 
\end{lem}

\begin{proof}
For $d\in \mathrm{D}(\textbf{s}), |g^{-1}(\{d\})\cap \mathrm{MF}^i(\textbf{s})|=1$ for all $1\le i\le c(\textbf{s})$. In fact, the element of each $g^{-1}(\{d\})\cap \mathrm{MF}^i(\textbf{s})$ can be obtained by cyclically permuting the tree components of the element of $g^{-1}(\{d\})\cap \mathrm{MF}^1(\textbf{s})$. This shows that $g^i$ is a bijection. The other two claims are straightforward.
\end{proof}
The map $g$ being surjective immediately gives the following result.
\begin{cor}\label{cor:marked forest and child sequence}
Let $\mathbb{MF}(\textbf{s})$ be a uniformly random element of $\mathrm{MF}(\textbf{s})$, then $g(\mathbb{MF}(\textbf{s}))$ is a uniformly random element of $\mathrm{D}(\textbf{s})$.
%
\end{cor}


\section{Concentration results}\label{sec:martingale stuff}
In the first part of this section, we deal with a martingale concerning the proportion of a fixed degree of uniformly permuted degree sequence. This will be useful for proving Proposition \ref{prop: first j converge} in Section \ref{sec:proof of prop6 and lem10} where we need to first show that the degree proportions in each large trees of $\mathcal{F}^{\downarrow}_\kappa$ are more or less in line with the degree proportion of the given degree sequences. The second part of this section deals with the variance bound of uniformly permuted child sequences, which leads to a key technical proposition on the height of tree components of $\mathbb{F}(\mathbf{s})$. For both subsections we will use concentration results from \cite{McDiarmid1998}.

Let $\mathbf{s}=(s^{(i)}, i\ge 0)$ with $|\mathbf{s}|=n$ be a fixed degree sequence and let $\mathbf{C}=(C_1, \cdots, C_{n})$ denote the uniformly permuted child sequence $\pi(d(\mathbf{s}))$ (recall the notation from Section \ref{sec:nto1}), where $\pi$ is a uniform random permutation of $[n]$. For each $i\ge 0$, let $q^{(i)}=s^{(i)}/n$ be the degree proportion of degree $i$ of $\mathbf{s}$.

\subsection{Martingales of degree proportions of uniformly permuted degree sequence}\label{sec:degree_proportion_martingale}
In this subsection, we introduce some martingales concerning proportions of particular degree appeared at each step in a uniformly permuted degree sequence and use them and martingale concentration inequality from \cite{McDiarmid1998} as tools to prove Lemma \ref{lem: bad event small} and Proposition \ref{prop: bad_characterization}, which are useful for eventually proving that the empirical degree distributions of large trees of $\mathbb{F}_\kappa$ behave well (Proposition \ref{prop:proportion and second moments convergence}).
We first recall the following martingale bound in \cite{McDiarmid1998}. Let $\{X_i\}_{i=0}^n$ be a bounded martingale adapted to a filtration $\{\mathcal{F}_i\}_{i=0}^n$. Let $V=\sum\limits_{i=0}^{n-1}var\{X_{i+1}~|~\mathcal{F}_i\},$ where
$$var\{X_{i+1}~|~\mathcal{F}_i\}:=\e{(X_{i+1}-X_i)^2~|~\mathcal{F}_i}=\e{X_{i+1}^2~|~\mathcal{F}_i}-X_i^2.$$ Let $$v=\mbox{ess sup }V, \mbox{ and }b=\max\limits_{0\le i\le n-1}\mbox{ess sup}(X_{i+1}-X_i~|~\mathcal{F}_i).$$
Then we have the following bound.
\begin{thm}[\cite{McDiarmid1998}, Theorem 3.15]\label{thm:concentration}
For any $t\ge 0$,
$$\p{\max\limits_{0\le i\le n} X_i\ge t}\le \exp \left(-\frac{t^2}{2v(1+bt\backslash(3v))}\right).$$
\end{thm}

For fixed $i$, for $0\le j\le n-1$, 
let $Y^{(i)}_j=|\{1\le l\le j: C_l=i\}|$ and let $X^{(i)}_j=s^{(i)}-Y^{(i)}_j$. Note that for $j>0$ 
\[X_j^{(i)}=\left\{
                  \begin{array}{ll}
                    X_{j-1}^{(i)}-1, & \mbox{ if }C_j=i; \\
                    X_{j-1}^{(i)}, & \mbox{ otherwise.}
                  \end{array}
                \right.\]

Let $\mathcal{F}_j$ be the $\sigma-$field generated by $C_1, \cdots, C_j$. 

\begin{lem}\label{lem:martingale}
Let $M_j^{(i)}:=\frac{X_j^{(i)}}{n-j}-q^{(i)}$, then

(a)\ $M_j^{(i)}$ is an $\mathcal{F}_j-$martingale;

(b)\ The {\em predictable quadratic variation} of $M_{j+1}^{(i)}$ satisfies

 $var\{M_{j+1}^{(i)}~|~\mathcal{F}_{j}\}:=\e{{M_{j+1}^{(i)}}^2~|~\mathcal{F}_{j}}-{M_{j}^{(i)}}^2\le \frac{1}{4}\frac{1}{(n-(j+1))^2}.$
\end{lem}
\begin{proof}
(a)\ Since $q^{(i)}$ is a constant, it suffices to show that $\frac{X_j^{(i)}}{n-j}$ is an $\mathcal{F}_j-$martingale. In fact
\begin{eqnarray*}
\e{X^{(i)}_{j+1}~|~\mathcal{F}_j} &=& X^{(i)}_j-\p{C_{j+1}=i~|~\mathcal{F}_j}\\
&=& X^{(i)}_j-\frac{X^{(i)}_j}{n-j},
\end{eqnarray*}
so \[\e{\frac{X_{j+1}^{(i)}}{n-(j+1)}~|~\mathcal{F}_j}=\frac{X^{(i)}_j}{n-j-1}(1-\frac{1}{n-j})=\frac{X^{(i)}_j}{n-j}.\]
Thus $\frac{X^{(i)}_j}{n-j}$ is an $\mathcal{F}_j-$martingale.

(b)\ By definition, we have
\begin{eqnarray*}
   & & var\{M_{j+1}^{(i)}~|~\mathcal{F}_{j}^{(i)}\}=\e{{M_{j+1}^{(i)}}^2~|~\mathcal{F}_{j}}-{M_{j}^{(i)}}^2 \\
&=& \frac{\e{{X_{j+1}^{(i)}}^2~|~\mathcal{F}_{j}}}{(n-(j+1))^2}-\frac{{X_{j}^{(i)}}^2}{(n-j)^2}
\end{eqnarray*}

Now we substitute
\[\e{{X_{j+1}^{(i)}}^2~|~\mathcal{F}_{j}}=(X_{j}^{(i)}-1)^2\frac{X_{j}^{(i)}}{n-j}+{X_{j}^{(i)}}^2\cdot\frac{n-j-X_{j}^{(i)}}{n-j}={X_{j}^{(i)}}^2-\frac{2{X_{j}^{(i)}}^2}{n-j}+\frac{X_{j}^{(i)}}{n-j}\]
in the above result and obtain 
\begin{eqnarray*}
var\{M_{j+1}^{(i)}~|~\mathcal{F}_{j}\}
&=& \frac{{X_{j}^{(i)}}^2}{(n-(j+1))^2}-\frac{{X_{j}^{(i)}}^2}{(n-j)^2}-\frac{2{X_{j}^{(i)}}^2}{(n-j)(n-(j+1))^2}+\frac{X_{j}^{(i)}}{(n-j)(n-(j+1))^2} \\
&=& \frac{(2(n-j)-1){X_{j}^{(i)}}^2}{(n-(j+1))^2(n-j)^2}-\frac{2{X_{j}^{(i)}}^2}{(n-j)(n-(j+1))^2}+\frac{X_{j}^{(i)}}{(n-j)(n-(j+1))^2}\\
&=& \frac{X_{j}^{(i)}(n-j-X_{j}^{(i)})}{(n-(j+1))^2(n-j)^2}\le  \frac{1}{4}\cdot\frac{1}{(n-(j+1))^2},
\end{eqnarray*}
which gives the claim.
\end{proof}
Now we can apply Theorem \ref{thm:concentration}.
\begin{prop}\label{prop: martingale_concentration}
For any $t>0$ and $0<s<n$, we have
\begin{equation}\label{eqn:concentration for all interval larger than log}
\p{\max\limits_{0\le j\le n-s}|q^{(i)}-\frac{X_{j}^{(i)}}{n-j}|\ge t}\le\exp\left(-\frac{3st^2}{3+2t}\right).
\end{equation} 
\end{prop}
\begin{proof}
Fix $s<n$, and consider the martingale $\{M^{(i)}_{j}\}_{j=0}^{n-s}$. By Lemma \ref{lem:martingale}(b), we know that
\[V=\sum\limits_{j=1}^{n-s}var\{M^{(i)}_{j}~|~\mathcal{F}_{j-1}\}\le\frac{1}{4}\sum\limits_{j=0}^{n-s-1}\frac{1}{(n-(j+1))^2}\le\frac{1}{4}\int\limits^{n-1}_{s-1}\frac{1}{x^2}dx\le\frac{1}{2s}.\] Hence $v=\mbox{ess sup }V\le \frac{1}{2s}$. Also, for $j\le n-s-1$, if $X^{(i)}_{j+1}=X^{(i)}_j$, then \[|M^{(i)}_{j+1}-M^{(i)}_j|=\frac{X^{(i)}_j}{(n-j)(n-j-1)}\le \frac{1}{s},\] and if $X^{(i)}_{j+1}=X^{(i)}_j-1$, then
\[|M_{j+1}^{(i)}-M_{j}^{(i)}|= |\frac{X_{j}^{(i)}-1}{n-(j+1)}-\frac{X_{j}^{(i)}}{n-j}|=|\frac{X_{j}^{(i)}}{(n-(j+1))(n-j)}-\frac{1}{n-(j+1)}|\le \frac{1}{s}.\]
Applying Theorem \ref{thm:concentration} to both $\{M^{(i)}_j\}^{n-s}_{j=0}$ and $\{-M^{(i)}_j\}^{n-s}_{j=0}$ gives $$\p{\max\limits_{0\le j\le n-s}\left|q^{(i)}-\frac{X_{j}^{(i)}}{n-j}\right|\ge t}\le\exp\left(-\frac{t^2}{\frac{1}{s}+\frac{2t}{3s}}\right),$$ as claimed.
\end{proof}
Now we give a probability bound of proportion of certain degree $i$ deviates from $q^{(i)}$ by an error of at least $\epsilon$.
\begin{lem}\label{lem: bad event small}
For fixed $i\in\mathbb{N}$ and $\epsilon>0$, let $B^{\epsilon,i}=\{\exists x\ge\log^3 n: |Y^{(i)}_{x}-q^{(i)} x|\ge\epsilon x\}$.
Then for any $n$ large enough such that $\frac{\sqrt{5}}{\log n}<\epsilon<1,~ \p{B^{\epsilon, i}} \le n^{-3}.$
\end{lem}
\begin{proof}
By symmetry, the event $\{\exists j\ge \log^3 n: |Y^{(i)}_{j}-q^{(i)} j|\ge\epsilon j\}$ has the same distribution as the event $\{\exists l\le n-\log^3 n: |X^{(i)}_{l}-q^{(i)}(n-l)|\ge\epsilon(n-l)\}$. Hence we can write \[\p{B^{\epsilon, i}}=\p{\max\limits_{0\le l\le n-\log^3 n} |q^{(i)}-\frac{X^{(i)}_{l}}{n-l}|\ge\epsilon}.\] 
Taking $s=\log^3 n, t=\epsilon$ in (\ref{eqn:concentration for all interval larger than log}), the result follows.
\end{proof}

Now we consider how degrees distribute among the tree components of the random forest $\mathbb{F}(\mathbf{s})$.
Write $\mathbb{F}(\mathbf{s})^{\downarrow}=(\mathbb{T}_{l},~ l\ge 1)$. Let $\mathbf{s}_{l}=(s^{(i)}_{l}, i\ge 0)$ denote the (empirical) degree sequence of the $l-$th largest tree $\mathbb{T}_{l}$, and let $\mathbf{n}_{l}=n(\mathbf{s}_{l})$. Recall that $q^{(i)}=s^{(i)}/n$ and let $q_{l}^{(i)}=s^{(i)}_{l}/\mathbf{n}_{l}$ be the empirical proportion of degree $i$ vertices of $\mathbb{T}_{l}$; if $\mathbb{F}(\mathbf{s})$ has fewer than $l$ trees then $q^{(i)}_l=0$. Note that $q^{(i)}$ is deterministic while $q_{l}^{(i)}$ is random. 
\begin{prop}\label{prop: bad_characterization}
For fixed $\epsilon>0$ and $i, l$, let $B^{\epsilon, i}_{l}=\{|q_{l}^{(i)}-q^{(i)}|>\epsilon\}$. Then for fixed $\epsilon>0, i\in\mathbb{N}$, we have \begin{equation}\label{eqn:bad event small}
\p{\bigcup\limits_{l:~|\mathbb{T}_l|>n^{1/4}}B^{\epsilon, i}_{l}}\le n\p{B^{\epsilon, i}}.
\end{equation} 
\end{prop}

\begin{proof}
Let $V$ be a uniformly random vertex of $\mathbb{F}(\mathbf{s})$, then $(\mathbb{F}(\mathbf{s}), V)$ is uniformly distributed in $\mathrm{MF}(\mathbf{s})$. List the nodes of $\mathbb{F}(\mathbf{s})$ in cyclic lexicographic order as $V=V_1, V_2, \cdots, V_n$, and for $i\le n$ let $C_i$ be the degree of $V_i$. By Corollary \ref{cor:marked forest and child sequence}, the sequence $(C_1, \cdots, C_n)=g(\mathbb{F}(\mathbf{s}), V)$ is uniformly distributed in $\mathrm{D}(\mathbf{s})$; in other words, it is distributed as a uniformly random permutation of $d(\mathbf{s})$. For any $1\le j\le n$, let $\tilde{B}^{\epsilon, i}_j$ be the event that there exists $m>n^{1/4}$ such that \[|\frac{\#\{1\le t\le m: C_{j+t~(\mathrm{mod}~n)}=i\}}{m}-q^{(i)}|>\epsilon.\] Since $(C_1, \cdots, C_n)$ is uniformly distributed in $\mathrm{D}(\mathbf{s})$, it is immediate that $\p{\tilde{B}^{\epsilon, i}_1}=\cdots=\p{\tilde{B}^{\epsilon, i}_n}$. Suppose a tree $T\in\mathbb{F}(\mathbf{s})$ with $|T|> n^{1/4}$ has that \[|\frac{\#\{u: k_T(u)=i\}}{|T|}-q^{(i)}|>\epsilon.\] If $V$ is not a node of $T$, then there exists $m>n^{1/4}, 0<j\le n-m$ such that \[V(T)=\{V_{j+1}, \cdots, V_{j+m}\},~ |\frac{\#\{1\le t\le m: C_{j+t}=i\}}{m}-q^{(i)}|>\epsilon.\] If $V$ is a node of $T$, then there exists $m>n^{1/4}, j>n-m$ such that \[V(T)=\{V_{j+1}, \cdots, V_{n}, V_1, \cdots, V_{j+m-n}\}, ~ |\frac{\#\{t\ge j+1 \mbox{ or } t\le j+m-n: C_t=i\}}{m}-q^{(i)}|>\epsilon.\]In either case we must have $\tilde{B}^{\epsilon, i}_j$ true for some $1\le j\le n$. Therefore \[\p{\bigcup\limits_{l:~|\mathbb{T}_l|>n^{1/4}}B^{\epsilon, i}_{l}}\le n\p{\tilde{B}^{\epsilon, i}_1}\le n\p{B^{\epsilon, i}},\] which gives the claim.
%
\end{proof}

\subsection{Probability bound of trees of random forest having abnormally large height}\label{sec:prob bound of having trees of abnormally large height}
In this subsection, we prove tail bounds on the heights of trees in $\mathbb{F}(\textbf{s})$, by first proving tail bounds on the sums of squares of the child sequences. This will be used in proving Proposition \ref{prop:diameter of small trees} in Section \ref{sec:proof of prop6 and lem10}. To be more specific, let $c=(c_1, c_2, \cdots, c_n)\in \mathrm{D}(\mathbf{s})$ be a child sequence with $\sigma^2(\textbf{s}):=\sum\limits_{i=1}^n c_i^2=\sum\limits_i i^2 s^{(i)}$ and write $M:=\sigma^2(\textbf{s})/n$ and $\Delta=\Delta(\mathbf{s}):=\max\limits_{i} c_i$. Recall that $C_1, C_2, \cdots, C_n$ are the uniformly permuted child sequence and let $S_j:=\sum\limits_{i\le j}C^2_i$. 
We will use the following theorem from \cite{McDiarmid1998}.
\begin{thm}[Theorem 2.7 in \cite{McDiarmid1998}]\label{thm:concentration_2_7}
Let random variables $X^\ast_1, \cdots, X^\ast_n$ be independent, with $X^\ast_k-\e{X^\ast_k}\le b$ for each $k$. Let $S^\ast_n=\sum X^\ast_k$, and let $S^\ast_n$ have expected value $\mu$ and variance $V$ (the sum of the variances of $X^\ast_k$). Then for any $t\ge 0$, with $\epsilon=bt/V$, we have
\[\p{S^\ast_n-\mu\ge t} \le \exp\left(-\frac{V}{b^2}((1+\epsilon)\ln(1+\epsilon)-\epsilon)\right) \le \exp\left(-\frac{t^2}{2V+2bt/3}\right).\]
\end{thm}
Since $C_1, C_2, \cdots, C_k$ are sampled without replacement from the population $c_1, c_2, \cdots, c_n$, we may not directly apply Theorem \ref{thm:concentration_2_7}. We address this issue as follows.

Recall (or see, e.g., \cite{Aldous1985}) that given real random variables $U, V$, we say $U$ is a \emph{dilation} of $V$ if there exist random variables $\hat{U}, \hat{V}$ such that \[\hat{U}\overset{d}{=}U,~ \hat{V}\overset{d}{=}V \mbox{ and } \e{\hat{U}|\hat{V}}=\hat{V}.\]

\begin{prop}[Proposition 20.6 in \cite{Aldous1985}]\label{prop:dilution}
Suppose $X_1, \cdots, X_k$ and $X^\ast_1, \cdots, X^\ast_k$ are samples from the same finite population $x_1,\cdots, x_n$, without replacement and with replacement, respectively. Let $S_k=\sum\limits_{i=1}^k X_i, S^\ast_k=\sum\limits_{i=1}^k X^\ast_i$. Then $S^\ast_k$ is a dilation of $S_k$. In particular, $\e{\phi(S^\ast_k)}\ge\e{\phi(S_k)}$ for all continuous convex function $\phi:\mathbb{R}\rightarrow\mathbb{R}$.
\end{prop}

The proof of Theorem \ref{thm:concentration_2_7}, in \cite{McDiarmid1998}, proceeds by bounding the quantity $\e{\exp(h(S^\ast_n-\mu))}$, where $h$ is any real number. By Proposition \ref{prop:dilution}, we have $\e{\exp(h(S_n-\mu))}\le\e{\exp(h(S^\ast_n-\mu)}$, which means that the proof applies mutatis mutandis in the setting of sampling without replacement.
\begin{cor}
Let $X_1, \cdots, X_k$ be samples from finite population $x_1,\cdots, x_n$, without replacement, with $X_1-\e{X_1}\le b$. Let $S_k=\sum\limits_{i=1}^k X_i, V=\sum\limits_{i=1}^k \mathrm{Var}(X_i)$ and $\mu_k=\e{S_k}$. Then for any $t\ge 0$, with $\epsilon=bt/V$, we have 
\begin{equation}\label{eqn:concentration_2_7}
\p{S_k-\mu_k \ge t} \le \exp\left(-\frac{V}{b^2}((1+\epsilon)\ln(1+\epsilon)-\epsilon)\right) \le \exp\left(-\frac{t^2}{2V+2bt/3}\right).
\end{equation}
\end{cor}

Now we get our probability bound on the deviations of $(S_k, k\le n)$.

\begin{proof}[Proof of Proposition \ref{prop: concentration on long intervals}]
We apply (\ref{eqn:concentration_2_7}); we have 
$\mu_k=\e{S_k}=\frac{k}{n}S_n, b=\Delta^2,$ \[V=\sum\limits_{i=1}^k \mathrm{Var}(C_i^2)\le k \e{C_1^4}=\frac{k}{n}\sum\limits_{i=1}^n c^4_i\le\frac{k}{n}\Delta^2\sigma^2(c)=k\Delta^2M,\] where $M=\sigma^2(c)/n$. For $\lambda>1$, taking $t=(\lambda-1)\frac{k}{n}\sigma^2(c)$, we obtain 
\begin{eqnarray*}
\p{S_k\ge\lambda\frac{k}{n}S_n}&=&\p{S_k-\mu_k\ge(\lambda-1)kM}\\
&\le&\exp\left(-\frac{((\lambda-1)kM)^2}{2k\Delta^2M+\frac{2}{3}\Delta^2(\lambda-1)kM}\right) 
\end{eqnarray*}

Using the assumption $\lambda\ge 2$ twice, we have \begin{eqnarray*}
\p{S_k\ge\lambda\frac{k}{n}S_n}&\le&\exp\left(-\frac{((\lambda-1)kM)^2}{\frac{8}{3}(\lambda-1)\Delta^2 kM}\right)\\
&=&\exp\left(-\frac{3(\lambda-1)kM}{8\Delta^2}\right)\le\exp\left(-\frac{3M}{16}\cdot\frac{\lambda k}{\Delta^2}\right)=\exp\left(-\frac{3\sigma^2(c)}{16n}\cdot\frac{\lambda k}{\Delta^2}\right),
\end{eqnarray*}
which finishes the proof.
\end{proof}

Using results from Section \ref{sec:nto1}, we now have the following estimate on variance of tree components of $\mathbb{F}(\mathbf{s})$.
For a tree $T$, we let $\sigma^2(T)=\sum\limits_{u\in T} k_T(u)^2$. 

\begin{prop}\label{prop:tree second momonet concentration alpha level}
Let $\textbf{s}=(s^{(i)}, i\ge 0)$ be a degree sequence with $|\textbf{s}|=n$ and $M=\sigma^2(\textbf{s})/n$. Then for $\lambda\ge 4, \alpha>\Delta^2(\textbf{s})/n$,
\begin{equation}\label{eqn:concentration_large_tree}
\p{\exists T \in \mathbb{F}(\textbf{s}): |T|\le \alpha n, \sigma^2(T)\ge\lambda \alpha \sigma^2(\textbf{s})}\le \frac{2}{\alpha}\exp(-\frac{3M}{16}\lambda).
\end{equation}
\end{prop}

\begin{proof}
Let $V$ be a uniformly random vertex of $\mathbb{F}(\mathbf{s})$, then $(\mathbb{F}(\mathbf{s}), V)$ is uniformly distributed in $\mathrm{MF}(\mathbf{s})$. List the nodes of $\mathbb{F}(\mathbf{s})$ in cyclic lexicographic order as $V=V_1, V_2, \cdots, V_n$, and for $i\le n$ let $C_i$ be the degree of $V_i$. By Corollary \ref{cor:marked forest and child sequence}, the sequence $(C_1, \cdots, C_n)=g(\mathbb{F}(\mathbf{s}), V)$ is uniformly distributed in $\mathrm{D}(\mathbf{s})$; in other words, it is distributed as a uniformly random permutation of $d(\mathbf{s})$. In what follows we omit some floor notations for readability. For $0\le j\le \lfloor\frac{1}{\alpha}\rfloor$, let $B_j$ be the event that \[\sum\limits_{i=j\alpha n+1}^{(j+2)\alpha n} C^2_{i~ (\mathrm{mod}~n)}\ge \lambda\alpha\sigma^2(\mathbf{s}).\] Since $C_1, \cdots, C_n$ is distributed as a uniformly random permutation of $d(\mathbf{s})$, we clearly have \[\p{B_0}=\p{B_1}=\cdots=\p{B_{\lfloor\frac{1}{\alpha}\rfloor}}.\]  
Suppose that a given tree $T\in\mathbb{F}(\mathbf{s})$ has $|T|\le \alpha n$ and $\sigma^2(T)\ge \lambda\alpha\sigma^2(\mathbf{s})$. Then there exist $0\le l< n$ and $m\le \alpha n$ such that $V(T)=\{V_{l+t~(\mathrm{mod}~n)}: 1\le t\le m\}$. Hence there exists $0\le j\le \lfloor \frac{1}{\alpha}\rfloor$ such that $V(T)\subset \{V_{i~(\mathrm{mod}~n)},~ j\alpha n+1\le i\le (j+2)\alpha n\}$. This implies that \[\sum\limits_{i=j\alpha n+1}^{(j+2)\alpha n} C^2_{i~(\mathrm{mod}~n)}\ge \sigma^2(T)\ge \lambda\alpha\sigma^2(\mathbf{s}),\] i.e. $B_j$ is true. Hence the probability in question is at most  \begin{eqnarray*}
(1+\lfloor\frac{1}{\alpha}\rfloor)\p{B_0}\le\frac{2}{\alpha}\p{S_{\lfloor2\alpha n\rfloor}\ge \lambda\alpha \sigma^2(\mathbf{s})} &\le& \frac{2}{\alpha}\p{S_{\lfloor2\alpha n\rfloor}\ge\frac{\lambda}{2}\cdot\frac{\lfloor2\alpha n\rfloor}{n}\sigma^2(\mathbf{s})}\\
&\le& \frac{2}{\alpha}\exp\left(-\frac{3M}{16}\lambda\right),
\end{eqnarray*}
where we take $k=\lfloor 2\alpha n\rfloor$ in Proposition \ref{prop: concentration on long intervals} and use $\alpha>\Delta^2(\textbf{s})/n$ at the last step.
\end{proof}

Now we finish this section by proving a key proposition on probability bound of $\mathbb{F}(\mathbf{s})$ containing trees with unusually large height. 

\begin{prop}\label{prop:key_prop}
$\forall~ \epsilon, \rho\in (0,1), \exists n_0=n_0(\epsilon)\in\mathbb{N}$ and $\beta_0>0$ such that the following is true. Let $\mathbf{s}$ be any degree sequence with $|\mathbf{s}|=n\ge n_0$. Suppose that $\Delta(\mathbf{s})\le n^{\frac{1-\epsilon}{2}}, s^{(1)}\le (1-\epsilon)|\mathbf{s}|$ and $\epsilon\le\sigma^2(\mathbf{s})/n\le 1/\epsilon$, then for any $0<\beta<\beta_0$, \[\p{\exists T\in\mathbb{F}(\textbf{s}): |T|<\beta n, h(T)>\beta^{1/8} n^{1/2}}\le \rho.\]
\end{prop}
\begin{proof}
Fix $\beta>0$ small, let $\delta=\beta^{1/8}$, and consider the following four events.
\begin{itemize}
\item $E_1$ is the event that there exists a tree $T$ (of $\mathbb{F}(\textbf{s})$) with $\Delta^2(\textbf{s})<|T|<\beta n$ and $\sigma^2(T)>(\frac{|T|}{n})^{1/2}\sigma^2(\textbf{s})$. 
\item $E_2$ is the event that there exists a tree $T$ with $|T|\le n^{1-\epsilon}$ and $\sigma^2(T)>n^{1-\frac{\epsilon}{2}}$. 
\item $E_3$ is the event that there exists a tree $T$ with $\Delta^2(\textbf{s})<|T|<\beta n$ and $\sigma^2(T)\le(\frac{|T|}{n})^{1/2}\sigma^2(\textbf{s})$ such that $h(T)>\delta n^{1/2}$.
\item $E_4$ is the event that there exists a tree $T$ with $|T|\le n^{1-\epsilon}$ and $\sigma^2(T)\le n^{1-\frac{\epsilon}{2}}$ such that $h(T)>\delta n^{1/2}$. 
\end{itemize}
If there is $T\in\mathbb{F}(\textbf{s})$ with $|T|<\beta n$, and $h(T)>\delta n^{1/2}$, then one of $E_1, E_2, E_3$ or $E_4$ must occur, so it suffices to bound $\p{E_1}+\p{E_2}+\p{E_3}+\p{E_4}$.
For $E_1$, we further decompose the interval $[\Delta^2(\mathbf{s}), \beta n]$ dyadically. In the next sum, we bound the $k-$th summand by taking $\alpha=\frac{\beta}{2^k}, \lambda=\frac{2^{\frac{k-1}{2}}}{\beta^{1/2}}\ge 4$ in Proposition \ref{prop:tree second momonet concentration alpha level}.
\begin{eqnarray}\label{eqn:p_E_1}
\p{E_1}&\le&\sum\limits_{k=0}^{\lfloor\log_2 \frac{\beta n}{\Delta^2(\mathbf{s})}\rfloor}\p{\exists T\in\mathbb{F}(\textbf{s}): |T|\in \left[\frac{\beta n}{2^{k+1}}, \frac{\beta n}{2^k}\right], \sigma^2(T)>\left(\frac{\beta}{2^{k+1}}\right)^{1/2}\sigma^2(\textbf{s})} \nonumber \\ 
&\le& \sum\limits_{k\ge 0}\frac{2^{k+1}}{\beta}\exp\left(-\frac{3\sigma^2(\mathbf{s})}{16n}\frac{2^{\frac{k-1}{2}}}{\beta^{1/2}}\right)\nonumber\\ 
&=& O\left(\frac{1}{\beta}\exp(-\frac{\epsilon}{\beta^{1/2}})\right)
\end{eqnarray} 
where we use that $\sigma^2(\mathbf{s})/n\ge\epsilon$ in the final line. 

Next, note that 
$\p{E_2}\le\sum\limits_{j=1}^{n^{1-\epsilon}}\p{\exists T\in\mathbb{F}(\mathbf{s}): |T|=j, \sigma^2(T)>n^{1-\epsilon/2}}$. 
For any fixed $j$, using Corollary \ref{cor:marked forest and child sequence}, with similar argument as in proof of Proposition \ref{prop:tree second momonet concentration alpha level}, we have 
\[\p{\exists T\in\mathbb{F}(\mathbf{s}): |T|=j, \sigma^2(T)>n^{1-\epsilon/2}}\le n\p{S_j\ge n^{1-\epsilon/2}}.\]
For any $j\le n^{1-\epsilon}$, use Proposition \ref{prop: concentration on long intervals} with $\lambda \frac{j}{n}\sigma^2(\mathbf{s})=n^{1-\epsilon/2}$ and $\Delta(\mathbf{s})\le n^{\frac{1-\epsilon}{2}}$, we have \[\p{S_j\ge n^{1-\epsilon/2}}\le \exp\left(-\frac{3\sigma^2(\mathbf{s})}{16n}\cdot \frac{\lambda j}{\Delta(\mathbf{s})^2}\right)\le \exp\left(-\frac{3}{16}n^{\epsilon/2}\right).\]
These give that 
\begin{equation}\label{eqn:p_E_2}
\p{E_2}\le n^{2-\epsilon}\exp\left(-\frac{3}{16}n^{\epsilon/2}\right).
\end{equation}

We bound $\p{E_3}$ as follows. For $k\ge 0$, let $E_{3,k}$ be the event that there exists $T\in\mathbb{F}(\mathbf{s})$  with $\frac{\beta n}{2^{k+1}}\le |T|\le \frac{\beta n}{2^k}$ and $\sigma^2(T)\le (\frac{|T|}{n})^{1/2}\sigma^2(\textbf{s})$ such that height $h(T)>\delta n^{1/2}$. Also, let $B$ be the event that there exists $T\in\mathbb{F}(\mathbf{s})$ with $|T|\ge n^{1/4}$ such that \[\left|\frac{s^{(1)}(T)}{|T|}-\frac{s^{(1)}}{n}\right|\ge\epsilon/2.\] For $n$ large enough, we have $\frac{\sqrt{5}}{\log n}<\epsilon/2<1$. Hence it is immediate from Lemma \ref{lem: bad event small} and Proposition \ref{prop: bad_characterization} that $\p{B}\le n^{-2}$ for $n$ large. Also, for $n$ large, if $h(T)\ge \delta n^{1/2}$ then $|T|\ge h(T)\ge n^{1/4}$, so
\begin{equation}\label{eqn:E_3_1}
\p{E_3}\le \p{B}+\sum\limits^{\lfloor\log_2 \frac{\beta n}{\Delta^2(\mathbf{s})}\rfloor}_{k=0}\p{E_{3,k}\cap B^c}\le \frac{1}{n^2}+\sum\limits^{\lfloor\log_2 \frac{\beta n}{\Delta^2(\mathbf{s})}\rfloor}_{k=0}\p{E_{3,k}\cap B^c}.
\end{equation}
Let $M$ be the number of trees $T\in\mathbb{F}(\mathbf{s})$
with $\frac{\beta n}{2^{k+1}}\le |T|\le \frac{\beta n}{2^k}$ and $\sigma^2(T)\le (\frac{|T|}{n})^{1/2}\sigma^2(\textbf{s})$, and list the {\em random} degree sequences of these trees as $\mathbf{R}_1, \cdots, \mathbf{R}_m$. Then for any degree sequences $\mathbf{r}_1, \cdots, \mathbf{r}_m$,
\begin{multline*}
\p{E_{3,k}\cap B^c\cap \{(\mathbf{R}_1,\cdots, \mathbf{R}_m)=(\mathbf{r}_1, \cdots, \mathbf{r}_m)\}}=\p{B^c\cap \{(\mathbf{R}_1,\cdots, \mathbf{R}_m)=(\mathbf{r}_1, \cdots, \mathbf{r}_m)\}}\\
\cdot\p{E_{3,k}~|~B^c\cap \{(\mathbf{R}_1,\cdots, \mathbf{R}_m)=(\mathbf{r}_1, \cdots, \mathbf{r}_m)\}}.
\end{multline*}
Moreover \[\p{E_{3,k}~|~B^c\cap \{(\mathbf{R}_1,\cdots, \mathbf{R}_m)=(\mathbf{r}_1, \cdots, \mathbf{r}_m)\}}=\p{\exists i\le m, h(\mathbb{T}(\mathbf{r}_i))\ge \delta n^{1/2}},\] where $\mathbb{T}(\mathbf{r}_i)$ is a uniformly random plane tree with degree sequence $\mathbf{r}_i$. It follows from these identities that
\begin{equation}\label{eqn:E_3_star}
\p{E_{3,k}\cap B^c}\le \sup \p{\exists i\le m, h(\mathbb{T}(\mathbf{r}_i))\ge \delta n^{1/2}},
\end{equation}
where the supremum is over vectors $(\mathbf{r}_1,\cdots,\mathbf{r}_m)$ of degree sequences such that \[\p{E_{3,k}\cap B^c \cap \{(\mathbf{R}_1,\cdots, \mathbf{R}_m)=(\mathbf{r}_1, \cdots, \mathbf{r}_m)\}}>0.\] The last condition implies that, for all $i\le m$, \[\left|\frac{\mathbf{r}^{(1)}_i}{n(\mathbf{r}_i)}-\frac{s^{(1)}}{n}\right|<\epsilon/2, \mbox{ so } \frac{\mathbf{r}^{(1)}_i}{n(\mathbf{r}_i)}<1-\epsilon/2,\] and that \[\sigma^2(\mathbf{r}_i)\le (\frac{n(\mathbf{r}_i)}{n})^{1/2}\sigma^2(\mathbf{s})\le (\frac{\beta}{2^k})^{1/2}\sigma^2(\mathbf{s}).\] Finally we must have $n(\mathbf{r}_i)\ge \frac{\beta}{2^{k+1}}n$ for all $i\le M$, so $M\le \frac{2^{k+1}}{\beta}$. Now recall Theorem \ref{thm:louigi_tree_height_tail_bound}, which states that for a degree sequence $\mathbf{r}=(r^{(i)}, i\ge 0)$ and for all $h\ge 1$, 
\[\p{h(\mathbb{T}(\mathbf{r}))\ge h}\le 7\exp\left(-h^2/608\sigma^2(\mathbf{r})1_{\mathbf{r}}^2\right)\]
where $1_{\mathbf{r}}=\frac{|\mathbf{r}|-2}{|\mathbf{r}|-1-r^{(1)}}$; note that this is at most $4/\epsilon$ for all degree sequences under consideration (for $n$ large enough such that $n^{1/4}\ge 4/\epsilon$).  
Using a union bound in (\ref{eqn:E_3_star}), and then applying Theorem \ref{thm:louigi_tree_height_tail_bound}, we obtain that
\[\p{E_{3,k}\cap B^c}\le \frac{2^{k+1}}{\beta}\cdot 7\exp\left(-\frac{\epsilon^3\delta^2}{9728}(\frac{2^k}{\beta})^{1/2}\right)\] where we use the assumption $\sigma^2(\mathbf{s})/n\le 1/\epsilon$. And summing over $k$ in (\ref{eqn:E_3_1}) yields that 
\begin{equation}\label{eqn:p_E_3}
\p{E_3}\le \sum\limits_{k\ge 0} \frac{2^{k+1}}{\beta}\cdot 7\exp\left(-\frac{\epsilon^3\delta^2}{9728}(\frac{2^k}{\beta})^{1/2}\right)+\frac{1}{n^2}\le C_5\frac{1}{\beta}\exp\left(-\frac{C_6}{\beta^{1/4}}\right)+\frac{1}{n^2}
\end{equation}
if we take $\delta=\beta^{1/8}$, where $C_5>0$ is some universal constant and $C_6>0$ is some constant depending on $\epsilon$.

For $\p{E_4}$, similar to the previous treatment of $\p{E_3}$, for $n$ large, we have \[\p{E_4}\le \frac{1}{n^2}+\p{E_4\cap B^c}.\] 
There are at most $n$ trees in total, so a reprise of the conditioning argument used to bound $\p{E_3}$ gives
\[\p{E_4\cap B^c}\le n\sup \p{h(\mathbb{T}(\mathbf{r}))\ge \delta n^{1/2}},\]
where the supremum is over degree sequences $\mathbf{r}$ with $n(\mathbf{r})\le n^{1-\epsilon}$, with $\sigma^2(\mathbf{r})\le n^{1-\epsilon/2}$, and with $r^{(1)}\le (1-\epsilon/2)n(\mathbf{r})$.
By Theorem \ref{thm:louigi_tree_height_tail_bound}, we obtain that
\begin{eqnarray}\label{eqn:p_E_4}
\p{E_4}\le\frac{1}{n^2}+7n\exp\left(-\frac{\delta^2 n}{608 \sigma^2(\mathbf{r})1^2_{\mathbf{r}}}\right) &\le& \frac{1}{n^2}+7n\exp\left(-\frac{\delta^2 n}{608n^{1-\frac{\epsilon}{2}}\frac{16}{\epsilon^2}}\right)\nonumber\\&=&\frac{1}{n^2}+7n\exp\left(-\frac{\epsilon^2}{9728}n^{\epsilon/2}\beta^{1/4}\right);
\end{eqnarray}
recall that we take $\delta=\beta^{1/8}$.
Of the bounds on $\p{E_i}, 1\le i\le 4$ in (\ref{eqn:p_E_1}), (\ref{eqn:p_E_2}), (\ref{eqn:p_E_3}) and (\ref{eqn:p_E_4}), the largest is for $\p{E_3}$ (provided $n$ is large enough). Hence by taking $\beta>0$ small enough, we can make the bound less than any prescribed number $\rho>0$, which yields the result.
\end{proof}

\section{Convergence of the Lukasiewicz walk of forest to first passage bridge}\label{sec:convergence of walk}

In this section, we aim to prove Theorem \ref{thm:walk convergence} and conclude Proposition \ref{prop:length convergence} as a corollary of Theorem \ref{thm:walk convergence}. Throughout the section, we fix a sequence $(\mathbf{s}_\kappa, \kappa\in \mathbb{N})$ of degree sequences, and let $\nk, \textbf{p}_\kappa$ be as in Section \ref{sec:intro} and the function $d$ be as in Section \ref{sec:nto1}. Write $\sigma_\kappa=\sigma(\textbf{p}_\kappa), d_\kappa=d(\sk), \sigma=\sigma(\textbf{p})$. Recall from Section \ref{sec:intro} that for $l\ge 0$, we write $B^{br}_l$ for the Brownian bridge of duration 1 from 0 to $-l$. Moreover, we simply write $B^{br}$ for the case $l=0$.

\begin{prop}\label{prop:walk convergence}
Assume $(\sk, \kappa\ge 0)$ satisfies the hypothesis of Theorem \ref{thm:forest of trees}, and in particular that $\ck=c(\sk)=(1+o(1))\lambda\sigma_\kappa\nk^{1/2}$ as $\kappa\rightarrow\infty$ for some $\lambda>0$ and that $\sigma_\kappa\rightarrow\sigma$. For each $\kappa\ge 0$, fix a uniform random permutation $\pi_\kappa$ of $[\nk]$, and define a $C[0,1]$ function $\widetilde{W}_\kappa$ by $$\widetilde{W}_\kappa(t):=\frac{W_{\pi_\kappa(d_\kappa)}(t\nk)}{\sigma_\kappa{\nk}^{1/2}}.$$ Then \[\widetilde{W}_\kappa \overset{d}{\to} B^{br}_{\lambda} \mbox{ in } C[0,1].\]
\end{prop}

To prove this theorem, we make use of the following result, which is Corollary 20.10 (a) in \cite{Aldous1985}.
\begin{thm}\label{thm:aldous_exchangeability}
Consider a triangular array $(Z_{q, i}: 1\le i\le M_q, 1\le q)$ of random variables satisfying

(a) For each $q$, the sequence $(Z_{q,1}, \cdots, Z_{q, M_q})$ is exchangeable;

(b) $\max\limits_i |Z_{q,i}|\overset{p}{\to} 0$ as $q \to\infty$.

Define $\mu_q=\sum\limits_i Z_{q,i},~ \tau_q^2=\sum\limits_i (Z_{q,i}-\frac{\mu_q}{M_q})^2$ and $S^q(t)=\sum\limits_{i=1}^{\lfloor tM_q\rfloor}Z_{q,i}$.

Let $X(t)=\tau B^{br}(t)+\mu t$ where $(\tau, \mu)$ is independent of $B^{br}$. Then
$$S^q\overset{d}{\to}X \mbox{ in } D[0,1] \mbox{ iff } (\mu_q, \tau_q)\overset{d}{\to}(\mu, \tau).$$
\end{thm}
\begin{proof}[Proof of Proposition \ref{prop:walk convergence}]
Let $d_{\kappa, i}:=\pi_\kappa(d_\kappa)_i-1,$ for $ 1\le i\le \nk$. Although $d_{\kappa, i}$ depends on $\kappa$, we will write $d_i$ instead of $d_{\kappa, i}$ from here for readability. We apply the above theorem directly with $Z_{\kappa, i}=\frac{d_i}{\sigma_{\kappa}\nk^{1/2}}$. Condition (a) is satisfied since $\pi_\kappa$ is a uniformly random permutation of $[\nk]$. Condition $(b)$ is satisfied since $\Delta_{\kappa}=o({\nk}^{1/2})$ and $\sup\sigma_{\kappa}<\infty$.

Next note that, since $\sum\limits_i d_i=\sum\limits_i (\pi_\kappa(d_\kappa)_i-1)=-\ck$,
\begin{equation}\label{eqn:correct mean}
\mu_\kappa=\sum\limits_i Z_{\kappa,i}=\frac{\sum d_i}{\sigma_{\kappa}\nk^{1/2}}=\frac{-\ck}{\sigma_{\kappa}\nk^{1/2}}\to -\lambda \mbox{\ as\ } \kappa\to\infty,
\end{equation}
the final convergence holding by our assumption on $\ck$. We also have
\begin{align*}
\tau_\kappa^2 & =\sum\limits_i \left(\frac{d_i}{\sigma_{\kappa}\nk^{1/2}}-\frac{-\ck}{\sigma_{\kappa}\nk^{1/2}\nk}\right)^2\\
& =\frac{1}{\sigma_{\kappa}^2\nk}(\sum\limits_i d_i^2+2\frac{\ck}{\nk}\sum\limits_i d_i+\frac{\ck^2}{\nk})
 =\frac{1}{\sigma_{\kappa}^2\nk}(\sum\limits_i d_i^2-\frac{\ck^2}{\nk})\\
& =\frac{1}{\sigma_{\kappa}^2\nk}\sum\limits_i d_i^2 + o(1),
\end{align*}
the last equation holding since $\ck=O(\nk^{1/2})$.

Next note that
\begin{align*}
\sum\limits_i d_i^2 = \sum\limits_i(\pi_\kappa(d_\kappa)_i-1)^2 &=\sum\limits_i ((d_\kappa)_i)^2+\nk-2\sum_i (d_\kappa)_i\\
&=\nk(\sigma_{\kappa}^2+1)+\nk-2(\nk-\ck)\\
&=\nk\sigma_\kappa^2+2\ck.
\end{align*}
It follows that
\begin{equation}\label{eqn: right sigma}
\tau_\kappa^2 = \frac{1}{\sigma_{\kappa}^2\nk}(\nk\sigma_\kappa^2+2\ck) + o(1)\rightarrow 1
\end{equation}
as $\kappa\rightarrow\infty$ by our assumption on $\mathbf{s}_\kappa$.

Using equations (\ref{eqn:correct mean}) and (\ref{eqn: right sigma}), by Theorem \ref{thm:aldous_exchangeability} we conclude that \[\left(\frac{W_{\pi_\kappa(d_\kappa)}(\lfloor t\nk\rfloor)}{\sigma_\kappa{\nk}^{1/2}}, ~0\le t\le 1\right)\overset{d}{\to} \left(B^{br}(t)-\lambda t, ~0\le t\le 1\right) \mbox{  in } D[0, 1].\] For all $t$, \[\Big|\frac{W_{\pi_\kappa(d_\kappa)}(\lfloor t\nk\rfloor)}{\sigma_\kappa{\nk}^{1/2}}-\frac{W_{\pi_\kappa(d_\kappa)}(t\nk)}{\sigma_\kappa{\nk}^{1/2}}\Big|\le \frac{\Delta_\kappa}{\sigma_\kappa\nk^{1/2}}=o(1)\] by assumption, so we must also have $\left(\widetilde{W}_\kappa(t), 0\le t\le 1\right) \overset{d}{\to} \left(B^{br}(t)-\lambda t,~ 0\le t\le 1\right)$ in $D[0, 1]$. Since the Skorohod topology relativized to $C[0,1]$ coincides with the uniform topology (see page 124 of \cite{Billingsley1999}), the result follows.
\end{proof}

Let $f: \mathcal{C}_0(1)\times[0, \infty)\to \mathcal{C}_0(1)$ be defined by $f(b, v):=\theta_u(b)$ where $u=\inf\{t: b(t)\le\min\limits_{0\le s\le 1} b(s)+v\}$. Note that since $b$ is continuous, the minimum of $b$ exists.
Also, for $v\le -\min\limits_{0\le s\le 1} b(s)$, we have $u=\inf\{t: b(t)=\min\limits_{0\le s\le 1} b(s)+v\}$ and for $v\ge -\min\limits_{0\le s\le 1} b(s)$ we have $u=0$ so $f(b,v)=\theta_0(b)=b$. 



Recall from Section \ref{sec:intro} the \emph{first passage bridge (of unit length from 0 to $-\lambda$)} $F^{br}_\lambda$ is $$(F^{br}_{\lambda}(t), 0\le t\le 1)\overset{d}{=}(B(t),0\le t\le 1 ~|~ T_{\lambda}=1)$$ where $T_{\lambda}:=\inf\{t:B(t)<-\lambda\}$ is the first passage time below level $-\lambda< 0$ and $B$ is the standard Brownian motion. We are going to use the following result from \cite{BCP2003}.

\begin{thm}[\cite{BCP2003}, Theorem 7]\label{thm: BCP 2003}
Let $\nu$ be uniformly distributed over $[0, \lambda]$ and independent of $B^{br}_{\lambda}$. Define the r.v. $U=\inf\{t: B^{br}_{\lambda}(t)=\inf_{0\le s\le 1}B^{br}_{\lambda}(s)+\nu\}$. Then the process $\theta_U(B^{br}_{\lambda})$ has the law of the first passage bridge $F^{br}_\lambda$. Moreover, $U$ is uniformly distributed over $[0, 1]$ and independent of $\theta_U(B^{br}_{\lambda})$.
\end{thm}

\begin{rmk}
Note that \cite{BCP2003} considers first passage times above positive levels, whereas we consider first passage below negative levels. But the two cases are clearly equivalent.
\end{rmk}
As preparation we begin with showing the almost sure continuity of the map $f$. We first show that for a fixed function $b$, the closeness of the location where $b$ is cyclically shifted will guarantee the continuity of the map $f$.

\begin{lem}\label{lem:boils down to shift at close place}
For any $b\in\mathcal{C}_0(1)$, the function $g^b: [0, 1]\rightarrow \mathcal{C}_0(1)$ with $g^b(u)=\theta_u(b)$ is uniformly continuous.
\end{lem}

\begin{proof}
We want to show that $\|\theta_u-\theta_v\|$ is small when $|u-v|$ is small. Since $\theta_u\circ\theta_v=\theta_{u+v \mod 1}$, without loss of generality, we can assume that $v=0$. In other words we just aim to bound $\|\theta_u(b)-b\|$ for small $u$. Fix $\delta\in(0, 1/2)$ and let $\epsilon=\epsilon(\delta)=\sup\limits_{|t-s|<\delta}|b(t)-b(s)|$ be the modulus of continuity of $b$. Let $0<u<\delta$. If $t\in [0, 1-u]$, then $|\theta_u(b)(t)-b(t)|=|b(t+u)-b(u)-b(t)|\le |b(u)-b(0)|+|b(t+u)-b(t)|\le 2\epsilon(u)$. If $t\in [1-u, 1]$, then $|\theta_u(b)(t)-b(t)|=|b(t+u-1)+b(1)-b(u)-b(t)|\le |b(t+u-1)-b(u)|+|b(1)-b(t)|\le 2\epsilon(u)$. Since $\epsilon(u)\rightarrow 0$ as $u\rightarrow 0$, the result follows.
\end{proof}

\begin{lem}\label{lem: local min}
Given $b\in\mathcal{C}_0(1)$ and $0\le v\le -\min(b)$, if $f(b, v)=\theta_{t_{v+\min(b)}}(b)$ is not continuous at $v$, then $b$ attains a local minimum at $t_{v+\min(b)}$.
\end{lem}

\begin{proof}
By Lemma \ref{lem:boils down to shift at close place}, if $f(b,v)$ is not continuous at $v$, then $t_{v+\min(b)}$ is not continuous at $v$. The continuity of $b$ clearly implies right-continuity of $t_{v+\min(b)}$ as a function of $v$. Moreover, for all $0\le v\le -\min(b)$, $b$ attains a left-local minimum at $t_{v+\min(b)}$. Letting $t^+=\lim\limits_{v'\uparrow v}t_{v'+\min(b)}$, then it follows that \[b(x)\ge v+\min(b) \mbox{ for all } x\in[t_{v+\min(b)}, t^+].\] This implies that if $t_{v+\min(b)}$ is not continuous at $v$, then $t^+>t_{v+\min(b)}$, so $b$ also attains a right-local minimum at $t_{v+\min(b)}$. This proves the lemma.
\end{proof}

For $\lambda>0$, we next collect a few properties of Brownian bridge $B^{br}_{\lambda}$ and first passage bridge $F^{br}_{\lambda}$:

\begin{lem}\label{lem: properties of brownian bridge}Brownian bridge $B^{br}_\lambda$ satisfies the following properties:

(a) Let $\tau_+=\inf\{t>0: B^{br}_{\lambda}(t)>0\},~ \tau_-=\inf\{t>0: B^{br}_{\lambda}(t)<0\}$, then almost surely $\tau_+=\tau_-=0$;

(b) Given two nonoverlapping closed intervals (which may share one common endpoint) in $[0, 1]$, the minima of $B^{br}_{\lambda}$ on these two intervals are almost surely different;

(c) Almost surely, every local minimum of $B^{br}_{\lambda}$ is a strict local minimum;

(d) The set of times where local minima are attained is countable.

Moreover, these four properties also hold for first passage bridge $F^{br}_\lambda$.
\end{lem}

\begin{proof}
First note that the four properties are satisfied by a standard Brownian motion $B$ (e.g. see Theorem 2.8 and Theorem 2.11 in \cite{MortersPeres2010}). Let $C_n$ be the set of functions $f\in C[0,1]$ such that all four properties in the lemma occur up to time $1-1/n$ (i.e. the restriction of $f$ on $[0, 1-1/n]$ satisfies all four properties). Then $\p{B\in C_n}=1$ for all $n\in\mathbb{N}$. By equation (\ref{eqn:brownian bridge change of measure}) and equation (\ref{eqn:first passage bridge change of measure}) we know that the law of $B^{br}_\lambda$ and the law of $F^{br}_\lambda$ are both absolutely continuous with respect to the law of $B$ up to time $1-1/n$. Hence we must have $\p{B^{br}_\lambda\in C_n}=\p{F^{br}_\lambda\in C_n}=1$ for any $n\in\mathbb{N}$. This immediately implies that properties (a), (c) and (d) hold for $B^{br}_\lambda$ and $F^{br}_\lambda$. It also implies (b), except for the case where one of the intervals has the form $[s,1]$ and the minimum on $[s, 1]$ is reached at 1. For $F^{br}_\lambda$, by definition the global minimum $-\lambda$ is uniquely achieved at 1, hence the minimum on $[s,1]$ will not be the same as the minimum on any nonoverlapping interval. For $B^{br}_\lambda$, consider $\tilde{B}_\lambda(t)=-B^{br}_\lambda(1-t)-\lambda$, then $\tilde{B}_\lambda\overset{d}{=}B^{br}_\lambda$, so $\tilde{B}_\lambda$ almost surely takes positive values on any interval $[0,\epsilon]$ by property (a). It follows that $\min\limits_{t\in[s,1]}B^{br}_\lambda(t)$ is almost surely achieved at some $t\neq 1$. This completes the proof. 
\end{proof}

\begin{lem}\label{lem: continuity of rotation}
Let $\nu$ be $Unif[0, \lambda]-$distributed and independent of $B^{br}_{\lambda}$. Then the function $f: \mathcal{C}_0(1)\times[0, \infty)\to \mathcal{C}_0(1)$ satisfies
$\p{f \mbox{ is continuous at } (B^{br}_{\lambda}, \nu)}=1$.
\end{lem}

\begin{proof}
By Lemma \ref{lem: local min}, we have $$\p{f \mbox{ is not continuous at } (B^{br}_{\lambda}, \nu)}\le \p{ B^{br}_{\lambda} \mbox{ attains a local minimum at } t_{\nu+\min(B^{br}_{\lambda})}}$$

Let $M=\{u\in[0, 1]: B^{br}_\lambda \mbox{ attains local minimum at } u\}$ and let $\tilde{M}=\{B^{br}_\lambda(u) : u\in M\}$. By Lemma \ref{lem: properties of brownian bridge}, $M$ is countable, hence $\tilde{M}$ is countable.

Next note that $\p{ B^{br}_{\lambda} \mbox{ attains a local minimum at } t_{\nu+\min(B^{br}_{\lambda})}}\le \p{\nu+\min(B^{br}_\lambda)\in \tilde{M}}$. Moreover, $\nu$ is a continuous random variable, independent of $B^{br}_\lambda$, so the last probability equals zero.
\end{proof}

Now we are ready to give the proof of Theorem \ref{thm:walk convergence}.

\begin{proof}[Proof of Theorem \ref{thm:walk convergence}]
For each $\kappa\ge 1$ let $\nu_{\kappa}$ be a uniformly random element of $[\ck]-1$ independent of $\pi_\kappa$, and let $\nu$ be $Unif[0, \lambda]$ and independent of $B^{br}_{\lambda}$. By Corollary \ref{cor:dfswalkofforest}, \[f(\widetilde{W}_\kappa,\frac{\nu_{\kappa}}{\sigma_{\kappa}{\nk}^{1/2}})=f\left(\frac{W_{\pi(d(\sk))}(t\nk)}{\sigma_{\kappa}{\nk}^{1/2}}, \frac{\nu_{\kappa}}{\sigma_{\kappa}{\nk}^{1/2}}\right)\overset{d}{=}\left(\frac{S_{\mathbb{F}_\kappa}(t\nk)}{\sigma_{\kappa}{\nk}^{1/2}}\right)_{t\in[0,1]}.\] By Proposition \ref{prop:walk convergence}, we have $\widetilde{W}_\kappa\overset{d}{\to} B^{br}_{\lambda}$, and clearly we have ${\sigma_{\kappa}}^{-1}{\nk}^{-1/2}\nu_{\kappa}\overset{d}{\to}\nu$. By independence we have $(\widetilde{W}_\kappa, {\sigma_{\kappa}}^{-1}{\nk}^{-1/2}\nu_{\kappa})\overset{d}{\to}(B^{br}_{\lambda}, \nu)$. Since by Lemma \ref{lem: continuity of rotation} we have \[\p{f \mbox{ is continuous at } (B^{br}_{\lambda}, \nu)}=1,\] we can apply the mapping theorem (e.g. Theorem 2.7 in \cite{Billingsley1999}) to conclude that $$f(\widetilde{W}_\kappa, {\sigma_{\kappa}}^{-1}{\nk}^{-1/2}\nu_{\kappa})\overset{d}{\to}f(B^{br}_{\lambda}, \nu).$$ By Theorem \ref{thm: BCP 2003}, $F^{br}_{\lambda}\overset{d}{=} f(B^{br}_{\lambda}, \nu)$, hence we conclude that \[\left(\frac{S_{\mathbb{F}_\kappa}(t\nk)}{\sigma_{\kappa}{\nk}^{1/2}}\right)_{t\in[0,1]}\overset{d}{\rightarrow}F^{br}_\lambda,\] as required.
\end{proof}


Now we begin with the preparation work to prove Proposition \ref{prop:length convergence}. We define the map $h: \mathcal{C}_0(1)\rightarrow l^{\downarrow}_1$ such that for $g\in\mathcal{C}_0(1),~ h(g)$ equals to the decreasing ordering of excursion length of $g(s)-\min\limits_{0\le s' < s}g(s')$. (we append at most countably many zeros to make $h(g)$ an element of $l^{\downarrow}_1$). Define $h_k: \mathcal{C}_0(1)\rightarrow \mathbb{R}^k$ as $h_k=\pi_k\circ h$ where $\pi_k: l^{\downarrow}_1\rightarrow \mathbb{R}^k$ is the projection onto the subspace spanned by the first $k$ coordinates.
%
%
To prove Proposition \ref{prop:length convergence}, we use the following result from \cite{ChassaingLouchard2002}.
\begin{lem}\label{lem: criteria for deterministic functions}[Lemma 3.8 and Corollary 3.10 in \cite{ChassaingLouchard2002}]
Suppose $\zeta:[0, 1]\rightarrow\mathbb{R}$ is continuous. Let $E$ be the set of non-empty intervals of $I=(l,r)$ such that \[\zeta(l)=\zeta(r)=\min\limits_{s\le l}\zeta(s),\ \ \ \ \zeta(s)>\zeta(l) \ \ \mbox{for }l<s<r. \]Suppose that for all intervals $(l_1, r_1), (l_2, r_2)\in E$ with $l_1<l_2$, we have \begin{equation}\label{eqn: condition in lemma}\zeta(l_1)>\zeta(l_2).\end{equation} Suppose also that the complement of $\cup_{I\in E}I$ has Lebesgue measure 0. Fix functions $(\zeta_m, m\ge 1)$ such that $\zeta_m\rightarrow \zeta$ uniformly on $[0,1]$, and real numbers $(t_{m,i},~ m, i\ge 1)$ which satisfy the following:

\ \ (i)\ \ $0=t_{m,0}<t_{m,1}<\cdots<t_{m,k}=1$;

\ \ (ii)\ $\zeta_m(t_{m,i})=\min\limits_{u\le t_{m,i}}\zeta_m(u)$;

\ \ (iii) $\lim_m \max_i (\zeta_m(t_{m,i})-\zeta_m(t_{m, i+1}))=0$.


Then the vector consisting of decreasingly ranked elements of $\{t_{m, i}-t_{m, i-1}: 1\le i\le k\}$ (attaching zeroes if necessary to make the vector an element in $\mathbb{R}^{|E|}$) converges componentwise and in $l_1$ to the vector consisting of decreasingly ranked elements of $\{r-l: (l,r)\in E\}$.
\end{lem}

\begin{lem}\label{lem: no same new min}
Let $\mathcal{E}$ be the set of excursions $\gamma$ of $ F^{br}_{\lambda}(s)-\min\limits_{0\le s'<s}  F^{br}_{\lambda}(s')$. Then almost surely for all $\gamma_1, \gamma_2\in\mathcal{E}$ with $l(\gamma_1)<l(\gamma_2)$, we have $F^{br}_{\lambda}(l(\gamma_1))> F^{br}_{\lambda}(l(\gamma_2))$.
\end{lem}

\begin{proof}
Suppose to the contrary that for some $\gamma_1, \gamma_2\in\mathcal{E}$ with $l(\gamma_1)<l(\gamma_2)$, we have $F^{br}_{\lambda}(l(\gamma_1))\le F^{br}_{\lambda}(l(\gamma_2))$, then since $\gamma_1, \gamma_2$ are excursions of $F^{br}_{\lambda}(s)-\min\limits_{0\le s'<s} F^{br}_{\lambda}(s')$, we must in fact have $F^{br}_{\lambda}(l(\gamma_1))= F^{br}_{\lambda}(l(\gamma_2))$. In this case then we can find $a, b, c\in\mathbb{Q}$ such that $a< l(\gamma_1)<b<l(\gamma_2)<c$, and $F^{br}_{\lambda}$ achieves the same minima (at $l(\gamma_1)$ and $l(\gamma_2)$ respectively) on $[a, b]$ and $[b, c]$. This has probability zero by Lemma \ref{lem: properties of brownian bridge} (b).
\end{proof}

To prove the next lemma, we introduce the following notation. Let $(S_{1/2}(\lambda), 0\le\lambda<\infty)$ denote a stable subordinator of index 1/2, which is the increasing process with stationary independent increments such that 
\[\e{\exp{(-\theta S_{1/2}(\lambda))}}=\exp{(-\lambda\sqrt{2\theta})}, \ \ \ \ \theta,\lambda\ge 0,\]
\[\p{S_{1/2}(1)\in dx}=(2\pi)^{-1/2}x^{-3/2}\exp{(-\frac{1}{2x})}dx, \ \ \ \ x>0.\]

\begin{lem}\label{lem: positive excursion length}
Almost surely, the coordinates of $h(F^{br}_{\lambda})$ sum to 1, and are all strictly positive.
\end{lem}

\begin{proof}
By Proposition 5 of \cite{BCP2003}, $h(F^{br}_{\lambda})$ has the law of the vector of ranked excursion lengths of $|B^{br}|$ conditioned to have total local time $\lambda$ at 0, which in turn has the same law as ranked excursion lengths of Brownian bridge conditioned to have total local time $\lambda$ at 0 (this vector has the same law as the random vector $Y(\lambda)$ in \cite{AldousPitman1998}, see equation (36) there). The latter is distributed as the scaled ranked jump sizes of the stable subordinator $S_{1/2}(\cdot)$ conditioned to be $\frac{1}{\lambda^2}$ at time 1 (e.g. see Theorem 4 in \cite{AldousPitman1998}). By Lemma 10 in \cite{AldousPitman1998}, the coordinates of $h(F^{br}_{\lambda})$ almost surely sum to 1. This immediately implies that the stable subordinator almost surely has infinitely many jumps, so almost surely all coordinates of $h(F^{br}_{\lambda})$ are strictly positive. Indeed, suppose to the contrary that the excursion intervals are $(l_1, r_1), \cdots, (l_k, r_k)$, where $r_i\le l_{i+1}, 1\le i\le k-1$. Then since $\sum\limits_{i=1}^k (r_i-l_i)=1$, we must in fact have $r_i=l_{i+1}, \forall 1\le i\le k-1$ and $l_1=0, r_k=1$. But this implies that $0=F^{br}_{\lambda}(l_1)=F^{br}_{\lambda}(r_1)=F^{br}_{\lambda}(l_2)=\cdots=F^{br}_{\lambda}(l_k)=F^{br}_{\lambda}(r_k)=F^{br}_{\lambda}(1)$, contradicting to the fact $F^{br}_{\lambda}(1)=-\lambda<0$.
\end{proof}

\begin{proof}[Proof of Proposition \ref{prop:length convergence}]
We first prove that for any fixed $j\ge 1$,
\begin{equation}\label{eqn: truncated tree sizes convergence} (|\mathbb{T}_{\kappa, l}|/\nk)_{1\le l\le j}\overset{d}{\rightarrow}(|\gamma_l|)_{1\le l \le j}.\end{equation}

Let $\zeta_\kappa=\left(\frac{S_{\mathbb{F}_\kappa}(t\nk)}{\sigma_\kappa{\nk}^{1/2}}\right)_{t\in[0,1]}$ and let $\zeta=\left(F^{br}_{\lambda}(t)\right)_{t\in[0,1]}$. By (\ref{eqn: walk convergence}) and by Skorokhod's representation theorem, we may work in a probability space in which $\zeta_\kappa\overset{a.s.}{\rightarrow}\zeta$. Let $E$ be the set of excursion intervals of $\zeta$. Then Lemma \ref{lem: no same new min} guarantees equation (\ref{eqn: condition in lemma}) in Lemma \ref{lem: criteria for deterministic functions} is true and Lemma \ref{lem: positive excursion length} guarantees that the complement of $\cup_{I\in E}I$ has Lebesgue measure 0, as required by Lemma \ref{lem: criteria for deterministic functions}. For each $\kappa$ let $t_{\kappa,0}=0$ and for $1\le j\le \ck$ let $t_{\kappa, j}$ be such that $\nk t_{\kappa, j}$ is the time the depth-first walk $S_{\mathbb{F}_\kappa}$ finishes visiting the $j-$th tree of $\mathbb{F}_\kappa$. Then almost surely, condition (i) of Lemma \ref{lem: criteria for deterministic functions} is clearly true and condition (iii) is also true since for each $1\le j\le \ck,~ \zeta_\kappa(t_{\kappa, j})=\zeta_\kappa(t_{\kappa, j-1})-\frac{1}{\sigma_\kappa\nk^{1/2}}$. The definition of Lukasiewicz walk guarantees that the times at which $\frac{S_{\mathbb{F}_\kappa}(t\nk)}{\sigma_\kappa{\nk}^{1/2}}$ hits a new minimum coincide with the times at which the walk finishes exploring the trees of the forest. Hence almost surely condition (ii) of Lemma \ref{lem: criteria for deterministic functions} is also satisfied. Also note that the vector consisting of decreasingly ranked elements of $\{t_{\kappa,j}-t_{\kappa,j-1}, 1\le j\le\ck \}$ is simply the scaled decreasing ordering of tree component sizes $(|\mathbb{T}_{\kappa,l}|/\nk)_{1\le l\le \ck}$. Hence by Lemma \ref{lem: criteria for deterministic functions} we know that  \[(|\mathbb{T}_{\kappa, l}|/\nk)_{1\le l\le j}\overset{a.s.}{\rightarrow}h_j(F^{br}_{\lambda})\] which immediately implies weak convergence. Lemma \ref{lem: positive excursion length} guarantees that this is true for any positive integer $j$. We also have $h_j(F^{br}_{\lambda})\overset{d}{=}(|\gamma_l|)_{1\le l\le j}$ by definition, and (\ref{eqn: truncated tree sizes convergence}) follows.

To prove (\ref{eqn: tree sizes convergence}) from (\ref{eqn: truncated tree sizes convergence}), we only need to prove that for any $\epsilon>0$, there exists $I_0\in\mathbb{N}$ such that $\limsup\limits_{\kappa\to\infty}\p{\sum\limits_{l> I_0}\frac{|\mathbb{T}_{\kappa,l}|}{\nk}>\epsilon}<\epsilon$. Since by Lemma \ref{lem: positive excursion length} we have $\sum\limits_l |\gamma_l|=1$ almost surely, in particular, $\lim\limits_{I\rightarrow \infty}\p{\sum\limits_{l>I}|\gamma_l|>\epsilon}=0$. So there exists $I_0$ such that $\p{\sum\limits_{l>I_0}|\gamma_l|>\epsilon}<\epsilon/2$. Let $A_\kappa$ be the event that $\sum\limits_{l\le I_0}\frac{|\mathbb{T}_{\kappa,l}|}{\nk}<1-\epsilon$ and $A$ be the event that $\sum\limits_{l\le I_0}|\gamma_l|<1-\epsilon$   (which has probability less than $\epsilon/2$ by our choice of $I_0$). By (\ref{eqn: truncated tree sizes convergence}), we have $|\p{A_\kappa}-\p{A}|<\epsilon/2$ for $\kappa$ large enough. Therefore
\begin{eqnarray*}
\limsup\limits_{\kappa\rightarrow\infty}\p{\sum\limits_{l>I_0}\frac{|\mathbb{T}_{\kappa,l}|}{\nk}>\epsilon} &=& \limsup\limits_{\kappa\rightarrow\infty}\p{A_\kappa}\\
&\le& \p{A}+ \limsup\limits_{\kappa\rightarrow\infty}|\p{A_\kappa}-\p{A}|\le\epsilon/2+\epsilon/2=\epsilon,
\end{eqnarray*}
as required.
\end{proof}

\section{Proof of Proposition \ref{prop: first j converge} and Proposition \ref{prop:diameter of small trees}}\label{sec:proof of prop6 and lem10}

We assume that we have the conditions of Theorem \ref{thm:forest of trees} hold. In particular, we have a probability distribution $\textbf{p}$ on $\mathbb{N}$. Recall that $\sigma=\sigma(\textbf{p}), \sigma_\kappa=\sigma(\textbf{p}_\kappa)$.
Let $\mathbf{s}_{\kappa,l}=(s^{(i)}_{\kappa,l}, i\ge 0)$ denote the degree sequence of $\mathbb{T}_{\kappa,l}$ and let $\mathbf{n}_{\kappa, l}=n(\mathbf{s}_{\kappa,l})$. Recall that $p_{\kappa}^{(i)}=s_\kappa^{(i)}/\nk$ and let $p_{\kappa,l}^{(i)}=s^{(i)}_{\kappa,l}/\mathbf{n}_{\kappa, l}$ be the empirical proportion of degree $i$ among all vertices of the $l-$th largest tree $\mathbb{T}_{\kappa, l}$. Note that $p_{\kappa}^{(i)}$ is deterministic while $p_{\kappa, l}^{(i)}$ is random. 

First, we are going to prove Proposition \ref{prop: first j converge} by using Theorem \ref{thm:Broutin-Marckert}. To do so, we will have to first show that the assumptions of Theorem \ref{thm:Broutin-Marckert} are satisfied in our setting.



\begin{prop}\label{prop:proportion and second moments convergence}
Under the assumption of Theorem \ref{thm:forest of trees}, for all $l\ge 1$, as $\kappa\rightarrow\infty$ we have

(a) $\textbf{p}_{\kappa,l}\overset{p}{\rightarrow} \textbf{p}$ coordinatewise, that is, $p^{(i)}_{\kappa,l}\overset{p}{\rightarrow} p^{(i)}$ for all $i\ge 1$.

(b) $\sigma(\textbf{p}_{\kappa,l})\overset{p}{\rightarrow}\sigma(\textbf{p})$.    
\end{prop}

\begin{proof}
For (a), we know that by Lemma \ref{lem: bad event small} and Proposition \ref{prop: bad_characterization}, for fixed $\epsilon>0, i,l\in\mathbb{N}$ and $\kappa$ large enough, we have \begin{equation}\label{eqn:bad event small_kappa version}
\p{|p_{\kappa, l}^{(i)}-p^{(i)}_\kappa|>\epsilon}\le 1/\nk+\p{|\mathbb{T}_{\kappa,l}|\le \nk^{1/4}}.
\end{equation} 

For any $\epsilon'>0$, there exists $\delta>0$ such that $\p{|\gamma_l|<\delta}<\epsilon'/2$ and by (\ref{eqn: truncated tree sizes convergence}) we can find $\kappa_0$ such that for all $\kappa\ge\kappa_0$ we have $\p{\frac{|\mathbb{T}_{\kappa,l}|}{\nk}<\delta}\le\p{|\gamma_l|<\delta}+\epsilon'/2$ and $\nk^{-3/4}<\delta$. Hence $\p{|\mathbb{T}_{\kappa,l}|\le \nk^{1/4}}=\p{\frac{|\mathbb{T}_{\kappa,l}|}{\nk}\le \nk^{-3/4}}\le\p{\frac{|\mathbb{T}_{\kappa,l}|}{\nk}\le \delta}<\epsilon'$. Hence $\p{|\mathbb{T}_{\kappa,l}|\le \nk^{1/4}}=o(1)$ as $\kappa\rightarrow\infty$. Therefore by (\ref{eqn:bad event small_kappa version}) we know that $|p^{(i)}_{\kappa,l}-p^{(i)}_\kappa|\overset{p}{\rightarrow} 0$ as $\kappa\rightarrow \infty$, which implies (a) since by assumption of Theorem \ref{thm:forest of trees} we have $\textbf{p}_\kappa$ converges to $\textbf{p}$ coordinatewise.

Now we proceed to prove (b). Fix $l\ge 1$ and $\delta>0$, and let $\epsilon>0$ be small enough that $$\limsup\limits_{\kappa\rightarrow\infty} \p{|\mathbb{T}_{\kappa,l}|<\epsilon\nk}<\delta.$$ Such $\epsilon$ exists by (\ref{eqn: truncated tree sizes convergence}).

Then let $M$ be large enough that $\sigma^2_{\kappa, >M}:=\sum\limits_{i>M}i^2 \frac{s_\kappa^{(i)}}{\nk}<\epsilon^2$ for all $\kappa$ (such $M$ exists since under the assumption of Theorem \ref{thm:forest of trees} $\sigma^2_\kappa$ converges). And let $\sigma^2_{\kappa,l,>M}=\sum\limits_{i>M}i^2 \frac{s^{(i)}_{\kappa,l}}{\mathbf{n}_{\kappa,l}}$ similarly. Note that $$\sigma^2_{\kappa,l,>M}\le \sum\limits_{i>M}i^2 \frac{s^{(i)}_\kappa}{|\mathbb{T}_{\kappa,l}|}=\sigma^2_{\kappa, >M}\frac{\nk}{|\mathbb{T}_{\kappa,l}|},$$ so if $\sigma^2_{\kappa,l,>M}>\epsilon$ then $|\mathbb{T}_{\kappa,l}|<\epsilon \nk$. By the triangle inequality, we have $$|\sigma^2(\textbf{p}_{\kappa,l})-\sigma^2(\textbf{p}_\kappa)|\le \sum\limits_{i\le M}i^2 |p^{(i)}_{\kappa,l}-p^{(i)}_\kappa|+\sum\limits_{i>M}i^2p^{(i)}_{\kappa,l}+\sum\limits_{i>M}i^2p^{(i)}_\kappa.$$
Since $|p^{(i)}_{\kappa,l}- p^{(i)}_\kappa|\rightarrow 0$ in probability for all $i$ by part (a), and $\sum\limits_{i>M} i^2 p^{(i)}_\kappa<\epsilon^2<\epsilon$, and $\sigma(\textbf{p}_\kappa)\rightarrow\sigma(\textbf{p})$ by assumption of Theorem \ref{thm:forest of trees}, this yields that $$\limsup\limits_{\kappa\rightarrow\infty}\p{|\sigma^2(\textbf{p}_{\kappa,l})-\sigma^2(\textbf{p})|>4\epsilon}\le \limsup\limits_{\kappa\rightarrow\infty}\p{\sum\limits_{i>M}i^2 p^{(i)}_{\kappa,l}>\epsilon}\le \limsup\limits_{\kappa\rightarrow\infty}\p{|\mathbb{T}_{\kappa,l}|<\epsilon\nk}<\delta,$$ which proves part (b).
\end{proof}

\begin{lem}\label{lem: max degree ok}
Let $\Delta_{\kappa,l}$ be the largest degree of a vertex of $\mathbb{T}_{\kappa,l}$. For any fixed $l$, we have \[\frac{\Delta_{\kappa,l}}{\sqrt{|\mathbb{T}_{\kappa,l}|}}\overset{p}{\rightarrow} 0 \mbox{ as } \kappa\to\infty.\]
\end{lem}

\begin{proof}
For any $\delta>0$, we need to prove $\lim\limits_{\kappa\to\infty}\p{\frac{\Delta_{\kappa,l}}{\sqrt{|\mathbb{T}_{\kappa,l}|}}>\delta}=0$. For any $\epsilon>0$, by Lemma \ref{lem: positive excursion length} we can choose $\epsilon'>0$ such that $\p{|\gamma_l|<\epsilon'}\le \epsilon/2$. Then choose $\kappa_0$ such that when $\kappa\ge\kappa_0$ we have \[\frac{\Delta^2_\kappa}{\nk}\cdot\frac{1}{\delta^2}<\epsilon' \mbox{ and } \p{\frac{|\mathbb{T}_{\kappa,l}|}{\nk}<\epsilon'}\le\p{|\gamma_l|<\epsilon'}+\frac{\epsilon}{2}.\] This is possible since $\Delta_\kappa=o(\nk^{1/2})$ by Remark \ref{rmk:degree assumption} and $|\mathbb{T}_{\kappa, l}|/\nk\overset{d}{\rightarrow} |\gamma_l|$ by (\ref{eqn: truncated tree sizes convergence}). Therefore \[\p{\frac{\Delta_{\kappa,l}}{\sqrt{|\mathbb{T}_{\kappa,l}|}}>\delta}\le \p{\frac{\Delta_\kappa}{\sqrt{|\mathbb{T}_{\kappa,l}|}}>\delta}=\p{\frac{|\mathbb{T}_{\kappa,l}|}{\nk}<\frac{\Delta^2_\kappa}{\nk}\cdot\frac{1}{\delta^2}}\le \p{\frac{|\mathbb{T}_{\kappa,l}|}{\nk}<\epsilon'}\le\epsilon,\] hence the claim.
\end{proof}

With Proposition \ref{prop:proportion and second moments convergence} and Lemma \ref{lem: max degree ok}, we are now ready to give the proof of Proposition \ref{prop: first j converge}.

\begin{proof}[Proof of Proposition \ref{prop: first j converge}]
Let $\mathbf{s}_{\kappa,l}$ be the random degree sequence of the $l-$th largest tree in the forest $\mathbb{F}_\kappa$. Then by Proposition \ref{prop:length convergence}, we have \[\left(\frac{n(\mathbf{s}_{\kappa,1})}{\nk},\cdots,\frac{n(\mathbf{s}_{\kappa,j})}{\nk}\right)\overset{d}\rightarrow \left(|\gamma_1|, \cdots, |\gamma_j|\right).\] By Proposition \ref{prop:proportion and second moments convergence} and Lemma \ref{lem: max degree ok}, we know we can apply Theorem \ref{thm:Broutin-Marckert} to $\mathcal{T}_{\kappa,l}$ to conclude that for each fixed $l\le j$, \[\frac{\nk^{1/2}}{n(\mathbf{s}_{\kappa,l})^{1/2}}\mathcal{T}_{\kappa,l}\overset{d}\rightarrow \mathcal{T}_{\mathbf{e}_l}\] where $(\mathbf{e}_l)_{l\le j}$ are independent copies of $\mathbf{e}$. Since the trees $(\mathcal{T}_{\kappa,l}, l\le j)$ are conditionally independent given their degree sequences, it follows that \[\left(\frac{\nk^{1/2}}{n(\mathbf{s}_{\kappa,l})^{1/2}}\mathcal{T}_{\kappa,l}, l\le j\right)\overset{d}\rightarrow\left(\mathcal{T}_{\mathbf{e}_l}, l\le j\right).\] The result follows by Brownian scaling.
\end{proof}


Finally, we give the proof of Proposition \ref{prop:diameter of small trees} based on Proposition \ref{prop:key_prop}, with the assumptions of Theorem \ref{thm:forest of trees_2}.

\begin{proof}[Proof of Proposition \ref{prop:diameter of small trees}]
By assumption we have $\sigma_\kappa\rightarrow\sigma\in(0,\infty)$ and $s^{(1)}_\kappa/|\sk|\rightarrow p^{(1)}<1$. Fix $\rho>0$ and let $\epsilon>0$ be such that $2\epsilon<\sigma^2<\frac{1}{2\epsilon}$. Then let $\beta_0=\beta_0(\rho,\epsilon)$ be as in Proposition \ref{prop:key_prop}, so that for all $n$ sufficiently large, if a degree sequence $\mathbf{s}$ satisfies $|\mathbf{s}|=n, \Delta(\mathbf{s})\le n^{\frac{1-\epsilon}{2}}, s^{(1)}\le (1-\epsilon)|\mathbf{s}|$ and $\epsilon\le \sigma^2(\mathbf{s})/n\le 1/\epsilon$, then for any $0<\beta<\beta_0$,  \[\p{\exists T\in\mathbb{F}(\textbf{s}): |T|<\beta n, h(T)>\beta^{1/8} n^{1/2}}\le \rho.\] For $\kappa$ sufficiently large, $\sk$ satisfies these conditions. Hence for any $0<\beta<\beta_0$, \begin{equation}\label{eqn: to ref in proof}
\p{\exists T\in\mathbb{F}(\sk): |T|<\beta\nk, h(T)>\beta^{1/8}\nk^{1/2}}\le \rho.
\end{equation}
Finally, taking $\beta=(a/\sigma_\kappa)^8$ in (\ref{eqn: to ref in proof}), since $\mathcal{T}_{\kappa, l}=\frac{\sigma_\kappa}{2\nk^{1/2}}\mathbb{T}_{\kappa,l}$ and for all $j>1/\beta$ we have $|\mathbb{T}_{\kappa,j}|<\beta\nk$, it follows that for all $\kappa$ sufficiently large, \begin{eqnarray*}
\p{\sup\limits_{l>j}h(\mathbb{T}_{\kappa,l})>\frac{a\nk^{1/2}}{\sigma_\kappa}}&\le& \p{\exists T\in\mathbb{F}(\sk): |T|<\beta\nk, h(T)>\beta^{1/8}\nk^{1/2}}\\
&\le& \rho.
\end{eqnarray*}
Since $\mathrm{diam}(\mathcal{T}_{\kappa,l})\le 2 h(\mathcal{T}_{\kappa,l})$, the result now follows easily. 
\end{proof}
\section{acknowledgements}
I would like to thank Louigi Addario-Berry for suggesting this project and numerous helpful discussions thereafter. This work was partially supported by NSERC CGS and I thank the institution.

\appendix
\section{proof of remark \ref{rmk:degree assumption}}
Let's restate remark \ref{rmk:degree assumption} as the following lemma.

\begin{lem}\label{lem:appendix degree assumption}
Suppose distributions $\textbf{p}_\kappa$ converges to $\textbf{p}$ coordinatewise and $\sigma(\textbf{p}_\kappa)\rightarrow \sigma(\textbf{p})\in(0,\infty)$ and $\frac{c(\sk)}{\nk^{1/2}}\rightarrow x\in(0,\infty)$, then $\mu(\textbf{p}_\kappa)\rightarrow \mu(\textbf{p})=1$ and $\Delta_\kappa/{\nk}^{1/2}\rightarrow 0$ as $\kappa\rightarrow\infty$.
\end{lem}
\begin{proof}
First, since $0\le\mu(\textbf{p})=\sum ip^{(i)}\le\sum i^2 p^{(i)}=\sigma^2(\textbf{p})<\infty$, we have $\mu(\textbf{p})\in(0,\infty)$. And we can compute the limit of $\mu(\textbf{p}_\kappa)$ explicitly: $$\mu(\textbf{p}_\kappa)=\sum ip^{(i)}_\kappa=\sum i\frac{s^{(i)}_\kappa}{\nk}=\frac{\nk-\ck}{\nk}\rightarrow 1$$ by our assumption of the magnitude of $\ck$.

Next, since $\mathbf{p}_\kappa\rightarrow \mathbf{p}$ coordinatewise, for all $M\in\mathbb{N}$ we have
\[\lim\limits_{\kappa\rightarrow\infty}|\sum\limits_{i\le M} i p_{\kappa}^{(i)}-\sum\limits_{i\le M} i p^{(i)}|=0.\] 
It follows that
\begin{eqnarray*}
\limsup\limits_{\kappa\rightarrow\infty} |\sum ip^{(i)}_\kappa-\sum ip^{(i)}| &=& \lim\limits_{M\rightarrow\infty}\limsup\limits_{\kappa\rightarrow\infty} |\sum\limits_{i\ge M} ip^{(i)}_\kappa-\sum\limits_{i\ge M} ip^{(i)}|\\
&\le & \lim\limits_{M\rightarrow\infty}\limsup\limits_{\kappa\rightarrow\infty} \left(\sum\limits_{i\ge M} ip^{(i)}_\kappa+\sum\limits_{i\ge M} ip^{(i)}\right)\\
&\le & \lim\limits_{M\rightarrow\infty}\limsup\limits_{\kappa\rightarrow\infty} \left(\sum\limits_{i\ge M} i^2p^{(i)}_\kappa+\sum\limits_{i\ge M} i^2p^{(i)}\right)\\
&=& 0,
\end{eqnarray*}
where the final equality holds since $\sigma(\mathbf{p})<\infty$ and $\sigma(\mathbf{p}_\kappa)\rightarrow\sigma(\mathbf{p})$. Hence $\mu(\mathbf{p}_\kappa)\rightarrow\mu(\mathbf{p})$.

Since $\textbf{p}_\kappa\rightarrow \textbf{p}$ coordinatewise, it follows that for any integer $N$, \[\limsup\limits_{\kappa\rightarrow\infty}|\sum\limits_{i\ge N}i^2 p^{(i)}_\kappa-\sum\limits_{i\ge N}i^2 p^{(i)}|=\limsup\limits_{\kappa\rightarrow\infty}|\sigma^2(\textbf{p}_\kappa)-\sigma^2(\textbf{p})|=0.\] Now let $\epsilon>0$ and let $N$ be large enough that $0<\sum\limits_{i\ge N}i^2p^{(i)}<\epsilon.$ Then for all $\kappa$ sufficiently large, $0<\sum
\limits_{i\ge N}i^2 p^{(i)}_\kappa<\epsilon$. But $\sum\limits_{i\ge N}i^2 p^{(i)}_\kappa\ge \epsilon\mathbbm{1}_{\Delta_\kappa\ge(\epsilon\nk)^{1/2}}$, so this implies that $\limsup\limits_{\kappa\rightarrow\infty}\frac{\Delta_\kappa}{\nk^{1/2}}\le\epsilon^{1/2}$. Since $\epsilon>0$ was arbitrary, the result follows.
\end{proof}
\section{proof of Remark \ref{rmk:GHP justification}}\label{sec:GHP justification}

The following proposition will be useful for our justification of Remark \ref{rmk:GHP justification} (see Lemma 2.4 in \cite{Le Gall2005} for a version dealing with Gromov-Hausdorff distance instead of Gromov-Hausdorff-Prokhorov distance):

\begin{prop}[Proposition 2.9 in \cite{Abraham2012}]\label{prop:distance coding function}
Let $f, g$ be two compactly supported non-negative continuous functions with $f(0)=g(0)=0$. Then $$d_{GHP}(\mathcal{T}_f,\mathcal{T}_g)\le 6||f-g||_\infty+|\sigma_f-\sigma_g|.$$
\end{prop}

Now we prove the following result.

\begin{prop}
The GH convergence in Theorem 1 in \cite{BroutinMarckert2012} can be strengthened to GHP convergence as in Theorem \ref{thm:Broutin-Marckert}.
\end{prop}

\begin{proof}
Let $C_\kappa$ be the contour function of $\mathbb{T}_\kappa$, define $\hat{C}_\kappa:[0, 1]\rightarrow [0,\infty)$ by letting $\hat{C}_\kappa(t)=\frac{\sigma(\textbf{p}_\kappa)}{2\nk^{1/2}}C_\kappa(2(\nk-1)t)$, then it is shown in \cite{BroutinMarckert2012} (see Theorem 3 there) that $\hat{C}_\kappa\overset{d}{\rightarrow}\mathbf{e}$ in the space $C([0,1],\mathbb{R})$, equipped with the supremum distance. By Proposition \ref{prop:distance coding function} and Skorokhod's representation theorem, it follows that $\mathcal{T}_{\hat{C}_\kappa}\overset{d}{\rightarrow}\mathcal{T}_{\mathbf{e}}$ in the GHP sense.

Next, metrically we may realize $\mathcal{T}_\kappa$ as the subspace of $\mathcal{T}_{\hat{C}_\kappa}$ consisting of the set $U$ of points whose distance from the root is an integer multiple of $\frac{\sigma(\textbf{p}_\kappa)}{2\nk^{1/2}}$. With this identification \[d_H(\mathcal{T}_\kappa, \mathcal{T}_{\hat{C}_\kappa})=\frac{1}{2}\cdot\frac{\sigma(\textbf{p}_\kappa)}{2\nk^{1/2}}.\] Moreover, the measure $\hat{\mu}_\kappa$ on $\mathcal{T}_{\hat{C}_\kappa}$ is the (normalized) length measure, and the measure $\mu_\kappa$ on $\mathcal{T}_\kappa$ is the uniform measure on its points. It follows that \[d_P(\hat{\mu}_\kappa, \mu_\kappa)\le \frac{1}{\nk}+ \frac{\sigma(\textbf{p}_\kappa)}{2\nk^{1/2}}.\] 
To see this, for each $u\in U$ which is not the root of $\mathcal{T}_\kappa$, let $e_u$ be the parent edge of $u$, which we view as a closed line segment of length $\epsilon=\frac{\sigma(\textbf{p}_\kappa)}{2\nk^{1/2}}$ in $\mathcal{T}_{\hat{C}_\kappa}$. For any non-empty set $S\subset U$, we have $\mu_\kappa(S)=|S|/\nk.$ Hence \[\hat{\mu}_\kappa(S^\epsilon)\ge \frac{|S|-1}{\nk-1}\ge\mu_\kappa(S)-\frac{1}{\nk},\] where the first inequality is because for non-root $u\in S$, we have $e_u\subset S^\epsilon$. On the other hand, let $A$ be a closed set in $\mathcal{T}_{\hat{C}_\kappa}$ and let $l=|\{e\in E(\mathcal{T}_\kappa): A\cap e\ne\emptyset\}|$. Then $A^\epsilon$ contains at least $l$ vertices of $\mathcal{T}_\kappa$ since no cycle exists, so \[\mu_\kappa(A^\epsilon)\ge \frac{l}{\nk}=\frac{l}{\nk-1}-\frac{l}{\nk(\nk-1)}\ge \hat{\mu}_\kappa(A)-\frac{1}{\nk}.\]
Hence $d_{GHP}(\mathcal{T}_\kappa,\mathcal{T}_{\hat{C}_\kappa})\overset{d}{\rightarrow} 0$. 
By the triangle inequality, it follows that $\mathcal{T}_\kappa\overset{d}{\rightarrow}\mathcal{T}_{\mathbf{e}}$ in the GHP sense.
\end{proof}

\end{document}